\documentclass[reqno]{amsart}
\usepackage{amsmath,amssymb}
\usepackage[mathscr]{euscript}
\usepackage{graphicx}
\usepackage{enumitem}
\usepackage[english]{babel}

\usepackage{soul} 

\usepackage[small]{caption}

\usepackage{tikz}

\newtheorem{theorem}{Theorem}[section]
\newtheorem{lemma}[theorem]{Lemma}
\newtheorem{proposition}[theorem]{Proposition}
\newtheorem{corollary}[theorem]{Corollary}
\newtheorem{asser}[theorem]{Assertion}

\newtheorem{remark}[theorem]{Remark}

\numberwithin{equation}{section}

\newcommand{\mc}[1]{{\mathcal #1}}
\newcommand{\ms}[1]{{\mathscr #1}}
\newcommand{\mf}[1]{{\mathfrak #1}}
\newcommand{\mb}[1]{{\mathbf #1}}
\newcommand{\bb}[1]{{\mathbb #1}}
\newcommand{\bs}[1]{{\boldsymbol #1}}

\renewcommand{\epsilon}{\varepsilon}
 
\newcommand{\cs}{\mathcal{S}}

\newcommand{\HH}{\mathbb{H}}

\newcommand{\PP}{\mathbb{P}}

\newcommand{\R}{\mathbb{R}}
\newcommand{\TT}{\mathbb{T}}

\let\b=\beta

\let\s=\sigma

\let\l=\lambda

\newcommand{\1}{\,\rlap{\small 1}\kern.13em 1}

\newcommand{\sqr}[2]{{\vcenter{\hrule height.#2pt%
                      \hbox{\vrule width.#2pt height#1pt\kern#1pt%
                            \vrule width.#2pt}%
                      \hrule height.#2pt}}}

\newcommand{\capa}{\textnormal{cap}}
\newcommand{\moins}{\textnormal{\textbf{-1}}}
\newcommand{\zero}{\textnormal{\textbf{0}}}
\newcommand{\plus}{\textnormal{\textbf{+1}}}

\begin{document}
	
\title[Metastability for the Blume-Capel model] {Metastability of the
  two-dimensional Blume-Capel model with zero chemical potential and
  small magnetic field on a large torus}

\author{C. Landim, P. Lemire, M. Mourragui}

\address{\noindent IMPA, Estrada Dona Castorina 110, CEP 22460 Rio de
  Janeiro, Brasil and CNRS UMR 6085, Universit\'e de Rouen, Avenue de
  l'Universit\'e, BP.12, Technop\^ole du Madril\-let, F76801
  Saint-\'Etienne-du-Rouvray, France.}
\email{landim@impa.br}

\address{
  \hfill\break\indent CNRS UMR 6085, Universit\'e de Rouen,
  \hfill\break\indent Avenue de l'Universit\'e, BP.12, Technop\^ole du
  Madril\-let, \hfill\break\indent
F76801 Saint-\'Etienne-du-Rouvray, France.}
\email{Paul.Lemire@etu.univ-rouen.fr}

\address{
  \hfill\break\indent CNRS UMR 6085, Universit\'e de Rouen,
  \hfill\break\indent Avenue de l'Universit\'e, BP.12, Technop\^ole du
  Madril\-let, \hfill\break\indent
F76801 Saint-\'Etienne-du-Rouvray, France.}
\email{Mustapha.Mourragui@univ-rouen.fr}

\begin{abstract} 
  We consider the Blume-Capel model with zero chemical potential and
  small magnetic field in a two-dimensional torus whose length
  increases with the inverse of the temeprature. We prove the
  mestastable behavior and that starting from a configuration with
  only negative spins, the process visits the configuration with only
  $0$-spins on its way to the ground state which is the configuration
  with all spins equal to $+1$.
\end{abstract}

\maketitle

\section{Introduction}
\label{intro}

The Blume--Capel model is a two dimensional, nearest-neighbor spin
system where the single spin variable takes three possible values:
$−1$, $0$ and $+1$. One can interpret it as a system of particles with
spins. The value $0$ of the spin at a lattice site corresponds to the
absence of particles, whereas the values $\pm 1$ correspond to the
presence of a particle with the respective spin.

Denote by $\TT_{L} = \lbrace 1,\dots,L \rbrace$ the discrete,
one-dimensional torus of length $L$, and let $\Lambda_{L} = \TT_{L}
\times \TT_{L}$, $\Omega_L = \lbrace -1, 0, 1
\rbrace^{\Lambda_L}$. Elements of $\Omega_L$ are called configurations
and are represented by the Greek letter $\sigma$. For $x \in
\Lambda_L$, $\s(x) \in \{-1,0,1\}$ stands for the value at $x$ of the
configuration $\s$ and is called the spin at $x$ of $\sigma$.

We consider in this article a Blume--Capel model with zero chemical
potential and a small positive magnetic field. Fix an external field
$0<h<2$, and denote by $\HH = \HH_{L,h} : \Omega_L \rightarrow \R$ the
Hamiltonian given by
\begin{equation} 
\HH (\s) = \sum \left( \s(y) - \s(x) \right)^2 - h \sum_{x \in
  \Lambda_L} \s(x), 
\label{hamiltonian}\end{equation}
where the first sum is carried over all unordered pairs of
nearest-neighbor sites of $\Lambda_L$.  We assumed that $h<2$ for the
configuration whose $0$-spins form a rectangle in a background of $-1$
spins to be a local minima of the Hamiltonian.

The continuous-time Metropolis dynamics at inverse temperature $\beta$
is the Markov chain on $\Omega_L$, denoted by
$\lbrace \s_t : t \geq 0 \rbrace$, whose infinitesimal generator
$L_\b$ acts on functions $f : \Omega_L \rightarrow \R$ as
\begin{equation*}
\begin{aligned} 
(L_\b f)(\s) &= \sum_{x \in \Lambda_L} R_\b (\s,\s^{x,+}) 
\, [ f(\s^{x,+}) - f(\s) ] \\
&+ \sum_{x \in \Lambda_L} R_\b (\s,\s^{x,-}) \, [ f(\s^{x,-}) - f(\s) ]\;.  
\end{aligned}
\end{equation*}
In this formula, $\sigma^{x,\pm}$ represents the configuration obtained
from $\sigma$ by modifying the spin at $x$ as follows,
\begin{equation*}
\sigma^{x,\pm} (z) := 
\begin{cases}
\sigma (x) \pm 1 \ \textrm{mod}\ 3 & \textrm{ if \ } z=x\;, \\
\sigma (z) & \textrm{ if \ } z\neq x\;, 
\end{cases}
\end{equation*}
where the sum is taken modulo $3$, and the jump rates $R_\b$ are given
by
\begin{equation*}
R_\beta(\sigma,\sigma^{x,\pm}) \, =\, 
\exp\Big\{ -\beta \, \big[\mathbb H (\sigma^{x ,\pm}) 
- \mathbb H (\sigma) \big]_+\Big\}\, , \quad x\in \Lambda_L\, ,
\end{equation*}
where $a_+$, $a\in \bb R$, stands for the positive part of $a$: $a_+ =
\max\{a, 0\}$.  We often write $R$ instead of $R_\b$.

Denote by $\mu_\b$ the Gibbs measure associated to the Hamiltonian
$\HH$ at inverse temperature $\b$,
\begin{equation} \mu_\b(\s) = \frac{1}{Z_\b} e^{-\b
    \HH(\s)}, 
\label{gibbsmeasure}
\end{equation} 
where $Z_\b$ is the partition function, the normalization constant
which turns $\mu_\b$ into a probability measure. We often write $\mu$
instead of $\mu_\b$.

Clearly, the Gibbs measure $\mu_\beta$ satisfies the detailed balance
conditions
\begin{equation*}
\mu_\beta (\sigma) \, R_\beta(\sigma,\sigma^{x,\pm})\, =\, 
\min\big\{\mu_\beta (\sigma)\, ,\, \mu_\beta (\sigma^{x, \pm}) \big\}
\, =\, \mu_\beta (\sigma^{x,\pm})\, R_\beta(\sigma^{x,\pm},\sigma)\, ,
\end{equation*}
$\sigma \in \Omega_L$, $x\in\Lambda_L$, and is therefore reversible
for the dynamics.

Denote by $\moins, \zero, \plus$ the configurations of $\Omega_L$ with
all spins equal to $-1,0,+1$, respectively. The configurations
$\moins$, $\zero$ are local minima of the Hamiltonian, while the
configuration $\plus$ is a global minimum. Moreover,
$\bb H(\textbf{0}) < \bb H(\textbf{-1})$.

The existence of several local minima of the energy turns the
Blume-Capel model a perfect dynamics to be examined by the theory
developed by Beltr\'an and Landim in \cite{bl2, bl7} for metastable
Markov chains.

Let $\ms M = \{{\bf -1}, {\bf 0}, {\bf +1}\}$, and denote by
$\Psi : \Omega_L\to \{-1,0, 1, \mf d\}$ the projection
defined by 
\begin{equation*}
\Psi(\sigma) \;=\; \sum_{\eta \in \ms M} \pi(\eta) \, \mb 1\{\sigma =
\eta\} 
\;+\; \mf d \, \mb 1\{\sigma \not\in \ms M \}\;,
\end{equation*}
where $\mf d$ is a point added to the set $\{-1,0, 1\}$ and
$\pi : \ms M \to \{-1,0,1\}$ is the application which provides the
magnetization of the states $\bf -1$, $\bf 0$, $\bf +1$:
$\pi({\bf -1}) = -1$, $\pi({\bf 0}) = 0$, $\pi({\bf +1}) = 1$.

A scheme has been developed in \cite{bl2, bl7, llm} to derive the
existence of a time-scale $\theta_\beta$ for which the the
finite-dimensional distributions of the hidden Markov chain
$\Psi(\sigma(t\theta_\beta))$ converge to the ones of a
$\{-1,0, 1\}$-valued, continuous-time Markov chain. Note that the
limiting process does not take the value $\mf d$.

The approach consists in proving first that in the time-scale
$\theta_\beta$ the trace of the process $\sigma_t$ on $\ms M$
converges in the Skorohod topology to a continuous-time Markov
chain. Then, to prove that in this time scale the time spent on
$\Omega_L \setminus \ms M$ is negligible. Finally, to show that, at
any fixed, the probability to be in $\Omega_L \setminus \ms M$ is 
negligible.

This is the content of the main result of this article. We are also
able to describe the path which drives the process from the highest
local minima, $\moins$ to the ground state $\plus$. We not only
characterize the critical droplet but we also describe precisely how
it grows until it invades the all space. In this process, we show that
starting from $\moins$, the model visits $\zero$ on its way to
$\plus$. 

We consider in this article the situation in which the length of the
torus increases with the inverse of the temperature. The case in which
$L$ is fixed has been considered by Cirillo and Nardi \cite{cn13}, by
us \cite{ll} and by Cirillo, Nardi and Spitoni \cite{cns17}.

The method imposes a limitation on the rate at which the space grows,
as we need the energy to prevail over the entropy created by the
multitude of configurations. In particular, the conditions on the
growth impose that the stationary state restricted to the valleys of
$\moins$, $\zero$ or $\plus$, defined at the beginning of the next
section, is concentrated on these configurations (cf. \eqref{8-7}).

The study of the metastability of the Blume--Capel model has been
initiated by Cirillo and Olivieri \cite{co96} and Manzo and Olivieri
\cite{mo01}. In this last paper the authors consider the metastable
behavior of the two-dimensional model in infinite volume with non-zero
chemical potential. The dynamics is, however, different, as flips from
$\pm 1$ to $\mp 1$ are not allowed.  We refer to these papers for the
interest of the model and its role in the understanding of
metastability.

\section{Notation and Results}
\label{sec1}

Denote by $D (\R_+, \Omega_L)$ the space of right-continuous functions
$\omega : \R_+ \rightarrow \Omega_L$ with left-limits and by $\PP_\s =
\PP_\s^{\b,L}$, $\s \in \Omega_L$, the probability measure on the path
space $D (\R_+,\Omega_L)$ induced by the Markov chain $(\sigma_t :
t\ge 0)$ starting from $\s$. Sometimes, we write $\sigma(t)$ instead
of $\sigma_t$.

Denote by $H_A$, $H_A^+$, $A \subset \Omega_L$, the hitting time of
$A$ and the time of the first return to $A$ respectively :
\begin{equation} 
\label{hittingtimes}
H_A \,=\, \inf \lbrace t > 0 : \s_t \in A \rbrace\;, \quad 
H_A^+ \,=\, \inf \lbrace t > T_1 : \s_t \in A \rbrace \;, 
\end{equation}
where $T_1$ represents the time of the first jump of the chain $\s_t$.

\smallskip\noindent{\bf Critical droplet.}
We have already observed that $\plus$ is the ground state of the
dynamics and that $\moins$ and $\zero$ are local minima of the
Hamiltonian. The first main result of this article characterizes the
critical droplet in the course from $\moins$ and $\zero$ to
$\plus$.  Let $n_0 = \lfloor 2/h \rfloor$, where $\lfloor a \rfloor$
stands for the integer part of $a \in \R_+$.

Denote by $\ms V_\moins$ the valley of $\moins$. This is the set
constituted of all configurations which can be attained from $\moins$
by flipping $n_0(n_0+1)$ or less spins from $\moins$. If after
$n_0(n_0+1)$ flips we reached a configuration where $n_0(n_0+1)$
$0$-spins form a $[n_0\times (n_0+1)]$-rectangle, we may flip one more
spin. Hence, all configurations of $\ms V_\moins$ differ from $\moins$
in at most $n_0(n_0+1)+1$ sites.

The valley $\ms V_\zero$ of $\zero$ is defined in a similar way, a
$n_0\times (n_0+1)$-rectangle of $+1$-spins replace the one of
$0$-spins. Here and below, when we refer to a $[n_0\times
(n_0+1)]$-rectangle, $n_0$ may be its length or its height.

Denote by ${\mf R}^l = {\mf R}^l_L$ the set of configurations with
$n_0 (n_0 +1) + 1$ $0$-spins forming, in a background of $-1$-spins, a
$n_0 \times (n_0 +1)$ rectangle with an extra $0$-spin attached to the
longest side of this rectangle. This means that the extra $0$-spin is
surrounded by three $-1$-spins and one $0$-spins which belongs to the
longest side of the rectangle. The set ${\mf R}^l_{0} = {\mf
  R}^l_{0,L}$ is defined analogously, the $-1$-spins, $0$-spins being
replaced by $0$-spins, $+1$-spins, respectively.

We show in the next theorem that, starting from $\moins$,
resp. $\zero$, the process visits ${\mf R}^l$, resp. ${\mf R}^l_0$,
before hitting $\{\zero, \plus\}$, resp. $\{\moins, \plus\}$. An
assumption on the growth of the torus is necessary to avoid the
entropy of configurations with high energy to prevail over the local
minima of the Hamiltonian. We assume that
\begin{equation}
\label{8-3}
\lim_{\b \to \infty} |\Lambda_L|\, e^{-2\b} \;=\; 0\;.
\end{equation}
We prove in Lemma \ref{p74} that under this condition,
\begin{equation}
\label{8-7}
\lim_{\b \to \infty} \frac{\mu_\beta(\ms V_\moins \setminus
\{\moins\})} {\mu_\beta (\moins)} \;=\; 0\;.
\end{equation}

\begin{theorem}
\label{mt22}
Assume that $0<h<1$, that $2/h$ is not an integer and that \eqref{8-3}
is in force. Then,
\begin{equation*}
\lim_{\b \to \infty} \PP_\moins [ H_{{\mf R}^l} < H_{\{\zero, \plus\}}] \;=\; 1 \;,
\quad \lim_{\b \to \infty} \PP_\zero [ H_{{\mf R}^l_0} < H_{\{\moins, \plus\}} ] \;=\; 1
\;.  
\end{equation*}
\end{theorem}

On the other hand, under the condition that
\begin{equation}
\label{8-4}
\lim_{\b \to \infty} 
|\Lambda_L|^{1/2}\, \big\{\, e^{-[(n_0+1)h-2]\beta} \,+\,   e^{-h\beta}\,\big\}
\quad\text{and}\quad \lim_{\b \to \infty} |\Lambda_L|^{2} \, e^{-(2-h) \beta} \;=\;0 \;,
\end{equation}
it follows from Proposition \ref{ml7} that
\begin{equation*}
\lim_{\b \to \infty} \inf_{\eta\in \mf R^l}
\PP_\eta [ H_{\{\zero, \plus\}} < H_\moins  ] \;>\; 0 \;,
\quad \lim_{\b \to \infty} \inf_{\xi\in \mf R^l_0}
\PP_\xi [ H_{\{\moins, \plus\}}  < H_\zero ] \;>\; 0 \;.  
\end{equation*}

The condition that $0<h<1$ ensures that the smaller side of the
critical rectangle is larger than or equal to $2$: $n_0\ge 2$.

\smallskip\noindent{\bf The route from $\moins$ to $\plus$.}
The second main result of the article asserts that starting from
$\moins$, the processes visits $\zero$ in its way to $\plus$.
Actually, in Section \ref{proofs1}, we describe in detail how the
critical droplet growths until it invades the all space.

\begin{theorem}
\label{mainprop}
Assume that $0<h<1$, that $2/h$ is not an integer and that condition
\eqref{8-4} is in force.  Then,
\begin{equation*}
\lim_{\b\to\infty} \PP_\moins [ H_\plus < H_\zero] \;=\;  0
\quad\text{and}\quad
\lim_{\b\to\infty} \PP_\zero [ H_\moins < H_\plus] \;=\;  0 \;.
\end{equation*}
\end{theorem}

The strategy of the proof relies on the assumption that while the
critical droplet increases, invading the entire space, nothing else
relevant happens in other parts of the torus. A new row or column is
added to a supercritical droplet at rate $e^{- (2-h) \b}/|\Lambda_L|$,
because $e^{- (2-h)\b }$ is the rate at which a negative spin is
flipped to $0$ when it is surrounded by three negative spins and one
$0$-spin, and $L$ is the time needed for a rate one asymmetric random
walk to reach $L$ starting from the origin. We need $L^2$ because we
need an extra $0$-spin to complete a row, and then to repeat this
procedure $L$ times for the droplet to fill the torus.  This rate has
to be confronted to the rate at which a $0$-spin appears somewhere in
the space. The rate at which a negative spin is flipped to $0$ when it
is surrounded by four negative spins is $e^{-(4-h) \b}$. Since this
may happen at $|\Lambda_L|$ different positions, the method of the
proof requires at least $|\Lambda_L| e^{- (4-h) \b}$ to be much
smaller than $e^{-(2-h) \b}/|\Lambda_L|$, that is, $|\Lambda_L|^{2}
e^{- 2\b } \to 0$. This almost explains the main hypothesis of the
theorem. The extra conditions appear because we need to take care of
other details to lengthy to explain here.


\smallskip\noindent{\bf Metastability.}
For two disjoint subsets $\ms A$, $\ms B$ of $\Omega_L$, denote by
$\capa (\ms A, \ms B)$ the capacity between $\ms A$ and $\ms B$: 
\begin{equation*} 
\capa (\ms A, \ms B) \;=\; \sum_{\sigma \in \ms A} \mu_\beta(\sigma)
\lambda_\beta (\sigma) \, \PP_\s [ H_{\ms B} < H_{\ms A}^+]\; ,
\end{equation*}
where
$\lambda_\beta (\sigma) = \sum_{\sigma'\in\Omega_L} R_\beta(\sigma,
\sigma')$ represents the holding times of the Blume-Capel model.  Let
\begin{equation}
\label{71}
\theta_\beta \;=\; \frac {\mu_\beta({\bf -1})}
{\capa({\bf -1}, \{\zero, {\bf +1}\})} \;\cdot
\end{equation}
We prove in Proposition \ref{mt3} that under the hypotheses of Theorem
\ref{mainprop} for any configuration $\eta\in\mf R^l$,
\begin{equation*}
\lim_{\beta\to\infty} \frac{\capa ({\bf -1},\{{\bf 0}, {\bf +1}\})}
{\mu_\beta (\eta)\, |\Lambda_L|}\;=\;\frac{4(2n_0+1)}3\, \;.
\end{equation*}
In particular,
\begin{equation}
\label{2-1}
\theta_\beta \;=\; \big[ 1+ o_\beta(1) \big]\, \frac{3}{4(2n_0+1)}\,
\frac 1{|\Lambda_L|} \, e^{a\, \beta}\;,
\end{equation}
where $a = \bb H(\eta) - \bb H(\moins) = 4(n_0+1) - [n_0(n_0+1) +1]$,
and $o_\beta(1)$ is a remainder which vanishes as $\beta\to\infty$.

\begin{theorem}
\label{mt1b}
Under the hypotheses of Theorem \ref{mainprop}, the finite-dimensional
distributions of the speeded-up, hidden Markov chain
$X_\beta(t) = \Psi\big(\sigma(\theta_\beta t)\big)$ converge to the
ones of the $\{-1,0, 1\}$-valued, continuous-time Markov chain $X(t)$
in which $1$ is an absorbing state, and whose jump rates are given by
\begin{equation*}
r(-1,0) \;=\;  r(0,1) \;=\; 1\;, \quad
r(-1,1) \;=\; r(0,-1)\;=\; 0\;. 
\end{equation*}
\end{theorem}

Note that the limit chain does not take the value $\mf d$, in contrast
with $X_\beta(t)$ since $\Psi(\sigma) = \mf d$ for all $\sigma\not\in
\ms M$.

A natural open question is the investigation of the dynamics in
infinite volume, extending the results of Manzo and Olivieri
\cite{mo01} to the Blume-Capel model with zero chemical potential. 

\smallskip The paper is organized as follows.  In the next section, we
recall some general results on potential theory of reversible Markov
chains and we prove a lemma on asymmetric birth and death chains which
is used later in the article. In Section \ref{energy}, we examine the
formation of a critical droplet and, in Section \ref{proofs1}, the
growth of a supercritical droplet. Theorems \ref{mt22} and
\ref{mainprop} are proved in Section \ref{sec3}. In the following two
sections, we prove that the trace of $\sigma_t$ on $\ms M$ converges
to a three-state Markov chain and that the time spent outside $\ms M$
is negligible. In the final section we prove Theorem \ref{mt1b}.

\section{Metastability of reversible Markov chains}
\label{sec2}

In this section, we present general results on reversible Markov
chains used in the next sections.  Fix a finite set $E$. Consider a
continuous-time, $E$-valued, Markov chain $\lbrace X_t : t \geq 0
\rbrace$. Assume that the chain $X_t$ is irreducible and that the
unique stationary state $\pi$ is reversible. 

Elements of $E$ are represented by the letters $x$, $y$. Let $\bb
P_x$, $x\in E$, be the distribution of the Markov chain $X_t$ starting
from $x$. Recall from \eqref{hittingtimes} the definition of the
hitting time and the return time to a set.

Denote by $R(x,y)$, $x\not = y\in E$, the jump rates of the Markov
chain $X_t$, and let $\l (x) = \sum_{y\in E} R (x,y)$ be the holding
rates. Denote by $p (x,y)$ the jump probabilities, so that
$R (x,y) = \l (x) \, p (x,y)$. The stationary state of the embedded
discrete-time Markov chain is given by $M (x) = \pi (x) \, \l (x)$.

\smallskip
\noindent{\bf Potential theory.}
Fix two subsets $A$, $B$ of $E$ such that $A \cap B =
\varnothing$.
Recall that the capacity between $A$ and $B$, denoted by $\capa(A,B)$,
is given by
\begin{equation} 
\label{defcapa}
\capa (A,B) \;=\; \sum_{x \in A} M (x) \, \PP_x [ H_B <
H_A^+]\; . 
\end{equation}

Denote by $L$ the generator of the Markov chain $X_t$ and by $D(f)$
the Dirichlet form of a function $f:E\to \bb R$:
\begin{equation*}
D(f) \;=\; -\, \sum_{x\in E} f(x) \, (Lf)(x)\, \pi(x)\;=\;
\frac 12\, \sum_{x,y} \pi (x)\, R (x,y)\, 
[f(y) - f(x) ]^2 \;.
\end{equation*}
In this later sum, each unordered pair $\{a, b\} \subset E$, $a\not =
b$, appears twice. The Dirichlet principle provides a variational
formula for the capacity:
\begin{equation} 
\label{dirichlet}
\capa(A,B) \;=\; \inf_f D(f) \;, 
\end{equation}
where the infimum is carried over all functions $f: E \rightarrow
[0,1]$ such that $f = 1$ on $A$ and $f = 0$ on $B$.

Denote by $\mc P$ the set of oriented edges of $E$: $\mc P =
\{(x,y)\in E\times E : R(x,y)>0\}$. An anti-symmetric function $\phi:
\mc P \to \mathbb{R}$ is called a flow. The divergence of a flow
$\phi$ at $x\in E$ is defined as
\begin{equation*}
(\mbox{div}\,\phi)(x) \;=\; \sum_{y: (x,y)\in \mc P} \phi (x,y)\;.
\end{equation*}
Let $\mc F_{A,B}$ be the set of flows such that
\begin{equation*}
\sum_{x\in A} (\mbox{div }\phi)(x) \;=\; 1\;, \quad
\sum_{y\in B} (\mbox{div }\phi)(y) \;=\; -\, 1\;, \quad
(\mbox{div }\phi)(z) \;=\; 0 \;, \quad z\not\in A\cup B\;.
\end{equation*}
The Thomson principle provides an alternative variational formula for
the capacity:
\begin{equation}
\label{thom}
\frac 1{\capa (A,B)} \;=\; \inf_{\phi \in \mc F_{A,B}} 
\frac{1}{2}\, \sum_{(x,y)\in \mc P} \frac 1{\pi(x)\, R(x,y)} \; \phi(x,y)^2
\;.
\end{equation}
We refer to \cite{bh15} for a proof of the Dirichlet and the Thomson
principles. 

In the Blume-Capel model, by definition of the rate function
$R_\beta(\sigma,\sigma^{x,\pm})$, 
\begin{equation*}
\mu_\beta (\sigma)\, R_\beta(\sigma,\sigma^{x,\pm}) \;=\;
\mu_\beta (\sigma) \wedge \mu_\beta (\sigma^{x,\pm})\;.
\end{equation*}
This identity will be used throughout the paper, without further
notice, to replace the left-hand side, which appears in the Dirichlet
and in the Thomson principle, by the right-hand side.

We turn to an estimate of hitting times in terms of capacities.  Fix
$x \in E \setminus (A \cup B)$. Then,
\begin{equation} 
\label{probastar}
\PP_x [ H_A < H_B ] \;=\; 
\frac{\PP_x [\, H_A < H_{B \cup \lbrace x \rbrace}^+\,]}
{\PP_x [\,H_{A \cup B} < H_x^+ \,]} \;\cdot
\end{equation}
Indeed, intersecting the event $\lbrace H_A < H_B \rbrace$ with
$\lbrace H_x^+ < H_{A \cup B} \rbrace$ and its complement, by the
Strong Markov property,
\begin{equation*}
\PP_x [H_A < H_B] \;=\; \PP_x [\,H_x^+ < H_{A \cup B}\,] \, \PP_x [H_A <
H_B] \;+\; \PP_x [\,H_A < H_{B \cup \lbrace x \rbrace}^+\,] \;,
\end{equation*}
which proves \eqref{probastar} by substracting the first term on the
right hand side from the left hand side.

Recall the definition of the capacity introduced in \eqref{defcapa}.
Multiplying and dividing the right hand side of \eqref{probastar} by
$M(x)$ yields that
\begin{equation*} 
\PP_x [H_A < H_B] \;=\; 
\frac{M(x) \, \PP_x [ \, H_A < H_{B \cup \lbrace x \rbrace}^+\, ]}
{\capa(x,A \cup B)} 
\;\leq\; \frac{M(x) \,\PP_x [\, H_A < H_x^+\, ]}
{\capa(x,A\cup B)} \;\cdot
\end{equation*} 
Therefore, by definition of the capacity and since, by
\eqref{dirichlet}, the capacity is monotone,
\begin{equation} 
\label{probatriangle}
\PP_x [H_A < H_B] \;\leq\; \frac{\capa(x,A)}
{\capa(x,A\cup B)} \;\leq\;
\frac{\capa(x,A)}{\capa(x,B)} \;\cdot 
\end{equation}
	
\smallskip
\noindent{\bf Trace process.}
We recall in this subsection the definition of the trace of a Markov
process on a proper subset of the state space. Fix $F\subsetneq E$ and
denote by $T_{F} (t)$ the time the process $X_t$ spent on the set $F$
in the time-interval $[0,t]$:
\begin{equation*}
T_F(t)
\;=\; \int_0^{t} \chi_F( X_s) \, ds\; ,
\end{equation*}
where $\chi_F$ represents the indicator function of the set $F$.
Denote by $S_F(t)$ the generalized inverse of the additive functional
$T_F(t)$:
\begin{equation*}
S_F(t) \;=\; \sup\{s\ge 0 : T_F(s) \le t\}\,. 
\end{equation*}
The recurrence guarantees that for all $t>0$, $S_F(t)$ is finite
almost surely. 

Denote by $X^F(t)$ the \emph{trace} of the chain $X_t$ on the set $F$,
defined by $X^F(t) := X(S_F(t))$.  It can be proven \cite{bl2} that
$X^F(t)$ is an irreducible, recurrent, continuous-time, $F$-valued
Markov chain. The jump rates of the chain $X^F(t)$, denoted by
$r_F(x,y)$, are given by
\begin{equation*}
r_F(x,y) \;=\; \lambda(x) \, \bb P_x\big[\, H^+_{F}
= H_y\, \big] \;, \quad x\,,\, y\,\in\, F\;, \quad
x\,\not=\, y \;.
\end{equation*}
The unique stationary probability measure for the trace chain, denoted
by $\pi_F$, is given by $\pi_F(x)=\pi(x)/\pi(F)$. Moreover, $\pi_F$ is
reversible if so is $\pi$.

\smallskip
\noindent{\bf Estimates of an eigenfunction.}
We derive in this subsection an estimate needed in the next sections.
Consider the continuous-time Markov chain $X_t$ on $E=\{0, \dots, n\}$
which jumps from $k$ to $k+1$ at rate $\epsilon$ and from $k+1$ to $k$
at rate $1$, $0\le k<n$.

Denote by $\mb P_k$ the distribution of the Markov chain $X_t$
starting from $k\in E$. Expectation with respect to $\mb P_k$ is
represented by $\mb E_k$.

Denote by $H_n$ the hitting time of $n$. Fix $\theta>0$, and let
$f:E\to \bb R_+$ be given by 
\begin{equation*}
f(k) \;=\; \mb E_k\big[ e^{-\theta\, H_n}\big]\;.
\end{equation*}
An elementary computation based on the strong Markov property shows
that $f$ is the solution of the boundary-valued elliptic problem 
\begin{equation*}
\left\{
\begin{aligned}
& (L f)(k) \;=\; \theta\, f(k)\;, \quad 0\le k<n\;, \\
& f(n) \;=\; 1\;,  
\end{aligned}
\right.
\end{equation*}
where $L$ stands for the generator of the Markov chain $X_t$.

\begin{lemma}
\label{ml3}
We have that $f(0) \le \epsilon^n/\theta$.
\end{lemma}

\begin{proof}
Multiplying the identity $(L f)(k) = \theta\, f(k)$ by $\epsilon^k$ and
summing over $0\le k<n$ yields that
\begin{equation*}
\sum_{k=0}^{n-1} \epsilon^{k+1} \, [f(k+1) - f(k)]  \;+\;
\sum_{k=1}^{n-1} \epsilon^k \, [f(k-1) - f(k)] \;=\;
\theta \sum_{k=0}^{n-1} f(k)\, \epsilon^k \;. 
\end{equation*}
On the left-hand side, all terms but one cancel so that
\begin{equation*}
\epsilon^{n} \, [f(n) - f(n-1)] \;=\;
\theta \sum_{k=0}^{n-1} f(k)\, \epsilon^k \;.
\end{equation*}
Since $f(k) \ge 0$ and $f(n)=1$, we have that
\begin{equation*}
\theta\, f(0) \;\le\; \theta \sum_{k=0}^{n-1} f(k)\, \epsilon^k \;=\; 
\epsilon^{n} \, [1 - f(n-1)]  \;\le\; \epsilon^{n}\;,
\end{equation*}
as claimed.
\end{proof}

This result has a content only in the case $\epsilon <1$, but we did
not use this condition in the proof.

\section{The emergence of a critical droplet}
\label{energy} 

In this section, we prove that starting from $\moins$, the process
creates a droplet of 0-spins on its way to $\lbrace \zero, \plus
\rbrace$, that is, a configuration $\s$ with a $n_0 \times (n_0 +1)$
rectangle of 0-spins (or \textit{0-rectangle}) and an extra 0-spin
attached to one of the sides of the rectangle, in a background of
negative spins. 

In the next section, we prove that if this extra $0$-spin is attached
to one of the longest sides of the rectangle, with a positive
probability the process hits $\zero$ before $\{\moins, \plus\}$, while
if it is attached to one of the shortest sides, with probability close
to $1$, the process returns to $\moins$ before hitting $\{\zero,
\plus\}$.  An important feature of this model is that the size of a
critical droplet is independent of $\b$ and $L$.
	
Throughout this section, $C_0$ is a large constant, which does not
depend on $\beta$ or $L$ but only on $h$, and whose value may change
from line to line. 

Recall the definition of the valley $\ms V_\moins$ introduced just
above equation \eqref{8-3}.  Note that there are few configurations in
$\ms V_\moins$ which differ from $\moins$ at $n_0(n_0+1)+1$
sites. Moreover, such configurations
\begin{equation}
\label{m05}
\text{may not have two spins equal to $+1$.}
\end{equation}

To define the boundary of the valley of $\ms V_\moins$, fix $L$ large,
and denote by $\mf B$ the set of configurations with $n_0(n_0+1)$
spins different from $-1$:
\begin{equation*}
\mf B \;=\; \big\{\sigma\in \Omega_L : 
|A(\sigma)| = n_0(n_0+1) \big\}\;,
\end{equation*}
where 
\begin{equation*}
A(\sigma) \;=\; \{x\in\Lambda_L : \sigma(x) \not = -1\}\;.
\end{equation*}
Denote by $\mf R$ the subset of $\mf B$ given by
\begin{equation*}
\mf R \;=\; \big\{ \sigma\in \{-1,0\}^{\Lambda_L} : 
A(\sigma) \text{ is a } n_0\times (n_0+1) \text{rectangle }\big\}
\;.
\end{equation*}
Note that the spins of a configuration $\sigma\in \mf R$ are either
$-1$ or $0$ and that all configurations in $\mf R$ have the same
energy.

Denote by $\mf R^+$ the set of configurations in $\Omega_L$ in which
there are $n_0(n_0+1)+1$ spins which are not equal to $-1$. Of these
spins, $n_0(n_0+1)$ form a $n_0\times (n_0+1)$-rectangle of $0$
spins. The remaining spin not equal to $-1$ is either $0$ or $+1$.
Figure \ref{fig1} present some examples of configurations in $\mf
R^+$. 

\begin{figure}[!h]
	\centering
	\begin{tikzpicture}[scale = .25]
	\foreach \y in {0, ..., 4}
	\foreach \x in {0, ..., 5}
	\draw[very thick] (\x, \y) -- (\x +1, \y) -- (\x
	+ 1, \y + 1) -- (\x , \y + 1) -- (\x , \y );
	\draw[very thick] (-1, 4) -- (0, 4) -- (0, 5) -- (-1 , 5) -- 
	(-1 , 4);
\fill[black!100] (-1,4) -- (0,4) -- (0,5) -- (-1,5) -- (-1,4);
	\foreach \x in {19, ..., 24}
	\foreach \y in {0, ..., 4}
	\draw[very thick] (\x, \y) -- (\x +1, \y) -- (\x
	+ 1, \y + 1) -- (\x , \y + 1) -- (\x , \y );
	\draw[very thick] (27, 1) -- (28, 1) -- (28, 2) -- (27, 2) -- (27, 1);
	\end{tikzpicture}
	\caption{Examples of two  configurations in ${\mf R}^+$ in the
          case where $n_0=5$. An empty (resp. filled) $1\times 1$
          square centered at $x$ has been placed at each site $x$
          occupied by a $0$-spin (resp. positive spin).  All the other
          spins are equal to $-1$.}
	\label{fig1}
\end{figure}
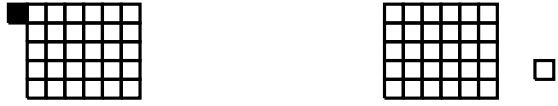

Let $\mf B^+$ be the boundary of $\ms V_\moins$. This set consists of
all configurations $\sigma$ in $\ms V_\moins$ which have a neighbor
[that is, a configuration $\sigma'$ which differs from $\sigma$ at one
site] which does not belong to $\ms V_\moins$. By definition of $\ms
V_\moins$, 
\begin{equation*}
\mf B^+ \;=\; (\mf B \setminus \mf R) \;\cup\;  \mf R^+ \;.
\end{equation*}

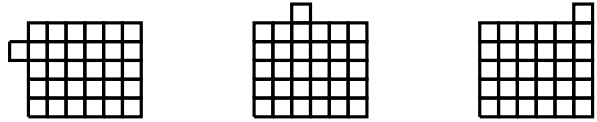
\begin{figure}[!h]
	\centering
	\begin{tikzpicture}[scale = .25]
	\foreach \y in {0, ..., 4}
	\foreach \x in {0, ..., 5}
	\draw[very thick] (\x, \y) -- (\x +1, \y) -- (\x
	+ 1, \y + 1) -- (\x , \y + 1) -- (\x , \y );
	\draw[very thick] (-1, 4) -- (0, 4) -- (0, 3) -- (-1 , 3) -- 
	(-1 , 4);
\foreach \y in {0, ..., 4}
	\foreach \x in {12, ..., 17}
	\draw[very thick] (\x, \y) -- (\x +1, \y) -- (\x
	+ 1, \y + 1) -- (\x , \y + 1) -- (\x , \y );
	\draw[very thick] (14, 5) -- (15, 5) -- (15, 6) -- (14 , 6) -- 
	(14 , 5);	
\foreach \y in {0, ..., 4}
	\foreach \x in {24, ..., 29}
	\draw[very thick] (\x, \y) -- (\x +1, \y) -- (\x
	+ 1, \y + 1) -- (\x , \y + 1) -- (\x , \y );
	\draw[very thick] (29, 5) -- (30, 5) -- (30, 6) -- (29 , 6) -- 
	(29 , 5);	
	\end{tikzpicture}
	\caption{Example of three configurations in ${\mf R}^a$ in the
          case where $n_0=5$. An $1\times 1$ empty square centered at
          $x$ has been placed at each site $x$ occupied by a $0$-spin.
          All the other spins are equal to $-1$. The one on the left
          belongs to $\mf R^s$. According to the notation introduced at the
          beginning of Section \ref{proofs1}, the one on the center
          belongs to $\mf R^{li}$ and the one on the right to $\mf
          R^{lc}$.}
	\label{fig2}
\end{figure}

Let $\mf R^a \subset \mf R^+$ be the set of configurations for which
the remaining spin is a $0$ spin attached to one of the sides of the
rectangle. Figure \ref{fig2} present some examples of configurations
in $\mf R^a$. We write the boundary $\mf B^+$ as 
\begin{equation}
\label{eq09} 
\mf B^+ \;=\; (\mf B \setminus \mf R) \;\cup\; (\mf R^+
\setminus \mf R^a) \;\cup\; \mf R^a \;. 
\end{equation}

Since $\mf B^+$ is the boundary of the valley $\ms V_\moins$, starting
from $\bf -1$, it is reached before the chain attains the set $\{{\bf
  0}, {\bf +1}\}$:
\begin{equation}
\label{ff02}
H_{\mf B^+} \;<\; H_{\{{\bf 0}, {\bf +1}\}} \quad 
\bb P_{\bf -1} \text{ a.s.}
\end{equation}

Note that all configurations of $\mf R^a$ have the same
energy and that $\bb H(\xi) = \bb H(\zeta)+2 -h$ if $\xi\in \mf R^a$,
$\zeta\in \mf R$. In particular, by Assertion 4.D in \cite{ll},
\begin{equation}
\label{eq07}
\bb H(\sigma) \;\ge\; \bb H(\xi) \;+\; h \;, \quad 
\sigma\in \mf B \setminus \mf R\;\;,\;\;
\xi\in \mf R^a\;.
\end{equation}
On the other hand, for a configuration $\eta\in\mf R^+ \setminus \mf
R^a$, $\bb H(\eta) \ge \bb H(\zeta)+4 -h$ if $\zeta\in \mf R$, so that
\begin{equation}
\label{eq08}
\bb H(\eta) \;\ge\; \bb H(\xi) \;+\; 2 \;, \quad 
\eta\in \mf R^+ \setminus \mf R^a\;\;,\;\;
\xi\in \mf R^a\;.
\end{equation}

In particular, at the boundary $\mf B^+$ the energy is minimized by
configurations in $\mf R^a$. This means that $\sigma_t$ should
attained $\mf B^+$ at $\mf R^a$. This is the content of the main
result of this section. Let
\begin{equation}
\label{m6}
\epsilon (\beta) \;=\; |\Lambda_L|\, e^{- 2 \b} \;+\; e^{-h \beta}\;.
\end{equation}

\begin{proposition}
\label{Lemma1}
There exists a finite constant $C_0$ such that
\begin{equation*}
\PP_\moins [ H_{\mf B^+} <  H_{\mf R^a} ] \;\le\; 
C_0\, \epsilon (\beta)
\end{equation*}
for all $\beta \ge C_0$.
\end{proposition}

\begin{proof}
The proof of this lemma is divided in several steps. Denote by
$\{\eta_t : t\ge 0\}$ the process obtained from the Blume-Capel model
by forbiding any jump from the valley $\ms V_\moins$ to its
complement. This process is sometimes called the reflected process.

It is clear that $\eta_t$ is irreducible and that its stationary
state, denoted by $\mu_{\ms V}$ is given by $\mu_{\ms V} (\sigma) =
(1/Z_{\ms V}) \, \exp\{ - \beta \bb H(\sigma)\}$, where $Z_{\ms V}$ is
a normalizing constant.

Moreover, starting from $\moins$, we may couple $\sigma_t$ with
$\eta_t$ in such a way that $\sigma_t = \eta_t$ until they hit the
boundary. Hence, if we denote by $\PP^{\ms V}_\moins$ the distribution
of $\eta_t$,
\begin{equation*}
\PP_\moins [ H_{\mf B^+} <  H_{\mf R^a} ] \;=\; \PP^{\ms V}_\moins [ H_{\mf
  B^+} <  H_{\mf R^a} ] \;.
\end{equation*}
By \eqref{probatriangle}, 
\begin{equation*}
\PP^{\ms V}_\moins [ H_{\mf B^+} < H_{\mf R^a}] \;=\; 
\PP^{\ms V}_\moins [ H_{\mf B^+ \setminus \mf R^a} < H_{\mf R^a}] 
\;\leq\; \frac{\capa_{\ms V}(\mf B^+ \setminus \mf R^a, \moins)}
{\capa_{\ms V}(\mf R^a, \moins)}\;,
\end{equation*}
where $\capa_{\ms V}$ represents the capacity with respect to the
process $\eta_t$. The lemma now follows from Lemmata \ref{ml4} and
\ref{fl1} below.
\end{proof}

Denote by $\Gamma_c$ the energy of a configuration $\s \in {\mf R}^a$ :
\begin{equation} 
\label{gammac}
\Gamma_c \;=\; 4\, (n_0 +1) \,-\,  h\, \big[\, n_0(n_0+1) \,+\, 1 \,-\,
|\Lambda_L| \, \big] \;. 
\end{equation}

\begin{lemma}	
\label{ml4}
There exists a finite constant $C_0$ such that
\begin{equation*}
\frac 1{\capa_{\ms V}(\mf R^a, \moins)} \;\leq\; C_0 \, 
\frac 1{|\Lambda_L|}\, Z_{\ms V}  \, e^{\b \Gamma_c} \;.
\end{equation*}
\end{lemma}

\begin{proof}
We use the Thomson principle to bound this capacity by constructing a
flow from $\moins$ to $\mf R_a$. The flow is constructed in two
stages. 

To explain the procedure we interpret a flow as a mass transport,
$\phi(\eta,\xi)$ representing the total mas transported from $\eta$ to
$\xi$. The goal is to define the transport of a mass equal to $1$ from
$\moins$ to $\mf R^a$. The first step consists in transferring the
mass from $\moins$ to $\mf R$.

\begin{figure}
\centering
\begin{tikzpicture}[scale = .2]
\foreach \y in {0, ..., 3}
\foreach \x in {0, ..., 3}
\draw[very thick] (\x, \y) -- (\x +1, \y) -- (\x
+ 1, \y + 1) -- (\x , \y + 1) -- (\x , \y );
%
\foreach \x in {8, ..., 11}
\foreach \y in {0, ..., 3}
\draw[very thick] (\x, \y) -- (\x +1, \y) -- (\x
+ 1, \y + 1) -- (\x , \y + 1) -- (\x , \y );
\draw[very thick] (8, 4) -- (9,4) -- (9,5) -- (8,5) -- 
(8,4);
\foreach \x in {16, ..., 19}
\foreach \y in {0, ..., 3}
\draw[very thick] (\x, \y) -- (\x +1, \y) -- (\x
+ 1, \y + 1) -- (\x , \y + 1) -- (\x , \y );
\foreach \x in {16, ..., 17}
\foreach \y in {4}
\draw[very thick] (\x, \y) -- (\x +1, \y) -- (\x
+ 1, \y + 1) -- (\x , \y + 1) -- (\x , \y );
\foreach \x in {24, ..., 27}
\foreach \y in {0, ..., 4}
\draw[very thick] (\x, \y) -- (\x +1, \y) -- (\x
+ 1, \y + 1) -- (\x , \y + 1) -- (\x , \y );
\foreach \x in {32, ..., 35}
\foreach \y in {0, ..., 4}
\draw[very thick] (\x, \y) -- (\x +1, \y) -- (\x
+ 1, \y + 1) -- (\x , \y + 1) -- (\x , \y );
\foreach \x in {35, ..., 36}
\foreach \y in {4}
\draw[very thick] (\x, \y) -- (\x +1, \y) -- (\x
+ 1, \y + 1) -- (\x , \y + 1) -- (\x , \y );

\foreach \x in {40, ..., 43}
\foreach \y in {0, ..., 4}
\draw[very thick] (\x, \y) -- (\x +1, \y) -- (\x
+ 1, \y + 1) -- (\x , \y + 1) -- (\x , \y );
\foreach \x in {41, ..., 44}
\foreach \y in {3,4}
\draw[very thick] (\x, \y) -- (\x +1, \y) -- (\x
+ 1, \y + 1) -- (\x , \y + 1) -- (\x , \y );
\foreach \x in {49, ..., 53}
\foreach \y in {0, ..., 4}
\draw[very thick] (\x, \y) -- (\x +1, \y) -- (\x
+ 1, \y + 1) -- (\x , \y + 1) -- (\x , \y );
\end{tikzpicture}
\caption{We present in this figure some configurations 
  $\zeta_{x,k}$ introduced in the proof of Lemma \ref{ml4}. Let
  $m=16$. The figure represent the configurations $\zeta_{x,m}$,
  $\zeta_{x,m+1}$, $\zeta_{x,m+2}$. Then, $\zeta_{x,m+4}$,
  $\zeta_{x,m+5}$, $\zeta_{x,m+6}$, and $\zeta_{x,m+9}$. An $1\times
  1$ empty square centered at $x$ has been placed at each site $x$
  occupied by a $0$-spin. All the other spins are equal to $-1$.}
\label{fig3}
\end{figure}
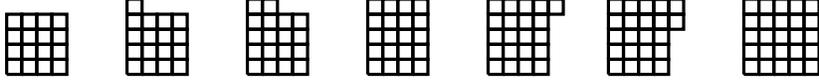

This is done as follows. Consider the sequence of points in $\bb Z^2$
which forms a succesion of squares of length $1$, $2$ up to $n_0$. It
is given by $u_1 = (1,1)$, $u_2 = (1,2)$, $u_3=(2,2)$, $u_4 = (2,1)$,
$u_5 = (1,3)$ and so on until $u_{n_0^2} = (n_0,1)$. Hence, we first
add a new line on the upper side of the square from left to right, and
then a new column on the right side from top to bottom. Once we
arrived at the $(n_0\times n_0)$-square, we add a final row at the
upper side of the square: Let $u_{n_0^2+k} = (k,n_0+1)$, $1\le k\le
n_0$, to obtain a $n_0\times(n_0+1)$-rectangle. 

Note that we reach through this procedure only rectangles whose height
is larger than the length. We could have defined flows which reach
both types of rectangles, but the bound would not improve
significantly. 

Let $A_k = \{u_1, \dots, u_k\}$ and denote by $A_{x,k}$ the set $A_k$
translated by $x\in \bb Z^2$: $A_{x,k} = x + A_k$. Denote by
$\zeta_{x,k}$ the configuration with $0$-spins at $A_{x,k}$ and
$-1$-spins elsewhere. Figure \ref{fig3} presents some of these
configurations. Let $\epsilon = 1/|\Lambda_L|$. The first stage of the
flow consists in transferring a mass $\epsilon$ from $\moins$ to each
$\zeta_{x,1}$ and then transfer this mass from $\zeta_{x,k}$ to
$\zeta_{x,k+1}$ for $1\le k<n_0(n_0+1)$.

Let $u_{n_0^2+n_0+1} = (1,n_0+2)$, and consider the configurations
$\zeta_{x,n_0^2+n_0+1}$ obtained through the correspondance adopted
above. The final stage consists in transferring the mass $\epsilon$
from $\zeta_{x,n_0^2+n_0}$ to $\zeta_{x,n_0^2+n_0+1}$. 

Since each configuration $\zeta_{x,n_0^2+n_0+1}$ belongs to $\mf R^a$,
the total effect of this procedure is to transport a mass equal to $1$
from the configuration $\moins$ to the set $\mf R^a$. 

Denote this flow by $\phi$, so that $\phi(\moins, \zeta_{x,1}) =
\phi(\zeta_{x,k} , \zeta_{x,k+1}) = \epsilon$. We extend this flow by
imposing it to be anti-symmetric and to vanish on the other bonds. It
is clear that this flow belongs to $\mc F_{\moins, \mf R^a}$, the set
of flows defined above \ref{thom}. Therefore, by the Thomson
principle,
\begin{equation*}
\frac 1{\capa_{\ms V}(\mf R^a, \moins)} \;\leq\; \epsilon^2\,
|\Lambda_L|\, \sum_{k=0}^{n_0(n_0+1)} 
\frac 1{\mu_{\ms V}(\zeta_{1,k})\, \wedge \mu_{\ms  V}(\zeta_{1,k+1})}
\;\cdot
\end{equation*}
Since $\epsilon = 1/|\Lambda_L|$ and $\mu_{\ms V}(\zeta_{1,k}) \ge
\mu_{\ms V}(\zeta_{1,n_0(n_0+1)+1})$, the previous expression is
bounded by
\begin{equation*}
\frac 1{\capa_{\ms V}(\mf R^a, \moins)} \;\leq\; C_0\, 
\frac 1{|\Lambda_L|} \, \frac 1{\mu_{\ms V}(\zeta_{1,n_0(n_0+1)+1})} \;,
\end{equation*}
which completes the proof of the lemma because the energy of the
configuration $\zeta_{1,n_0(n_0+1)+1}$ is $\Gamma_c$.
\end{proof}

We turn to the upper bound for $\capa(\mf B^+ \setminus \mf R^a,
\moins)$. 

\begin{lemma}
\label{fl1}
There exists a finite constant $C_0$ such that
\begin{equation*}
\capa_{\ms V}(\moins,\mf B^+ \setminus \mf R^a) 
\;\le \;  C_0\, \frac{1}{Z_{\ms V}} \, |\Lambda_L|\, e^{-\b \Gamma_c} \, \Big\{
|\Lambda_L|\, e^{- 2 \b} \,+\, e^{-h \beta} \Big\} \;.
\end{equation*}
for all $\beta \ge C_0$.
\end{lemma}

The proof of this lemma is divided in several steps. Let $B := \mf B^+
\setminus \mf R^a$ and let $\chi_{B} : \Omega_L \to \bb R$ be the
indicator of the set $B$. Since $\chi_B({\bf -1})=0$ and $\chi_B
(\sigma)=1$ for $\sigma\in B$, by the Dirichlet principle,
\begin{equation}
\label{ff2}
\capa_{\ms V}(\mf B^+ \setminus \mf R^a, {\bf -1}) \;\le\; D_{\ms V}(\chi_{B})\;,
\end{equation}
where $D_{\ms V}(f)$ stands for the Dirichlet form of $f$ for the
reflected process $\eta_t$.  An elementary computation yields that
\begin{equation}
\label{ff3}
D_{\ms V}(\chi_B) \;=\; \sum_{\sigma \in B}
\sum_{\sigma' \in \ms V_\moins \setminus B} \mu_{\ms V}(\sigma) 
\land \mu_{\ms V}(\sigma')\;,
\end{equation}
where the second sum is carried over all configurations $\sigma'$
which belong to $\ms V_\moins \setminus B$ and which differ from
$\sigma$ at exactly one spin. We denote this relation by $\sigma' \sim
\sigma$.

Let
\begin{equation}
\label{B+}
B_1 \;:=\; \mf B \,\setminus\, \mf R\;, \quad
B_2 \;:=\; \mf R^+ \,\setminus\, \mf R^a\;,
\end{equation}
so that $B = B_1 \cup B_2$, and consider separately the sums over
$B_1$ and $B_2$. We start with $B_2$.

\begin{asser}
\label{fasA}
We have that
\begin{equation*}
\sum_{\sigma \in B_2}
\sum_{\sigma'\sim\sigma} \mu_{\ms V}(\sigma) \land \mu_{\ms V}(\sigma') 
\;\le \;  C_0\, \frac{1}{Z_{\ms V}} \, |\Lambda_L|\, e^{-\b \Gamma_c} \, \Big\{
|\Lambda_L|\, e^{- 2 \b} \,+\, e^{-[10-h]\beta} \Big\} \;.
\end{equation*}
\end{asser}

\begin{proof}
A configuration $\eta \in B_2$ has a $n_0 \times (n_0 +1)$ rectangle
of 0-spins, and an extra-spin. If this extra-spin is attached to the
rectangle it is equal to $+1$, while it may be $0$ or $+1$ if it is
not. We study the two cases separately.
	
Fix a configuration $\eta \in B_2$ where the extra-spin is attached to
the rectangle, so that $\HH(\eta) = \Gamma_c + (10 - h)$. Consider a
configuration $\s' \in \ms V_\moins \setminus B$ such that $\s' \sim
\eta$. As $\s'$ may have at most $n_0(n_0+1) +1$ spins different from
$-1$, this excludes the possibility that $\s'$ is obtained from $\eta$
by flipping a $-1$. By \eqref{m05}, configurations in $\mf B$ with
$n_0(n_0+1) +1$ spins different from $-1$ may not have two spins equal
to $+1$. This excludes flipping a $0$ to $+1$. Finally, we may not
flip the $+1$ to $0$ because by doing so we obtain a configuration in
$\mf R^a$, and thus not in $B = \mf B^+\setminus \mf R^a$.

Hence, either $\s'$ is obtained from $\eta$ by flipping the $+1$ to
$-1$, or it is obtained by flipping a $0$ to $-1$. In the first case
$\bb H(\sigma') < \bb H(\eta)$, while in the second case if the
$0$-spin belongs to the corner, $\bb H(\sigma') > \bb H(\eta)$. Since
the number of configurations obtained by these flips is bounded by a
finite constant, the contribution to the sum appearing in the
statement of the assertion is bounded by
\begin{equation*}
C_0\, \sum_{\eta} \mu_{\ms V}(\eta) \;\le \;  C_0\, \frac{1}{Z_{\ms V}}\, 
e^{-\beta\, \Gamma_c} \, |\Lambda_L| \, e^{-[10-h]\beta}\;,
\end{equation*}
where the factor $|\Lambda_L|$ comes from the number of possible
positions of the rectangle, while the constant $C_0$ absorbs the
number of positions of the positive spin.

Fix now a configuration $\eta \in B_2$ where the extra-spin is not
attached to the rectangle. Then $\HH(\eta) \geq \Gamma_c +
2$. Consider a configuration $\s' \in \ms V_\moins \setminus B$ such
that $\s' \sim \eta$. As before, since $\s'$ may have at most
$n_0(n_0+1) +1$ spins different from $-1$, this excludes the
possibility that $\s'$ is obtained from $\eta$ by flipping a $-1$.  By
excluding this possibility, we are left with a finite number
[depending on $n_0$] of possible jumps. Hence, the contribution of
configurations of this type to the sum appearing in the statement of
the assertion of is bounded by
\begin{equation*}
C_0\, \sum_{\eta} \mu_{\ms V}(\eta) \;\le \;
C_0\, \frac{1}{Z_{\ms V}} \, e^{-\b \Gamma_c} \, |\Lambda_L|^2\,
e^{- 2 \b} \;,
\end{equation*}
where the factor $|\Lambda_L|^2$ appeared to take into account the
possible positions of the rectangle and of the extra particle.
This proves the assertion.
\end{proof}

It remains to examine the sum over $B_1$.  Denote by $N(\sigma)$ the
number of spins of the configuration $\sigma$ which are different from
$-1$:
\begin{equation}
\label{ff1}
N(\sigma) \;=\; \# A(\sigma)
\;=\; \# \{x\in \Lambda_L : \sigma_x \,\not =\, -1\}\;.
\end{equation}
Next assertion states that we can restrict our attention to
configurations $\s$ which have no spin equal to $+1$. For a
configuration $\s$ such that $N(\s) \leq n_0 (n_0 +1) +1,$ let $\s^o$
be the one obtained from $\s$ by replacing all spins equal to +1 by
0-spins : $\s^o_x = \s_x \land 0$.

\begin{asser}
\label{fasB}
For all $\sigma\in \Omega_L$ such that $N(\s) \leq n_0 (n_0 +1) +1$, 
\begin{equation*}
\bb H (\sigma^o) \;\le\; \bb H (\sigma) \;.
\end{equation*}
\end{asser}

\begin{proof}
This result is clearly not true in general because $\bf +1$ is the
ground state. It holds because we are limiting the number of spins
different from $-1$.

For a configuration $\sigma\in\Omega_L$, denote by $I_{a,b}(\sigma)$,
$-1\le a < b \le 1$, the number of unordered pairs $\{x,y\}$ of
$\Omega_L$ such that $\Vert x-y\Vert = 1$, $\{\sigma_x, \sigma_y\} =
\{a,b\}$, where $\Vert z\Vert$ stands for the Euclidean norm of
$z\in\bb R^2$.
	
An elementary computation yields that
\begin{equation*}
\bb H (\sigma) \;-\; \bb H (\sigma^o) \;=\;
I_{0,1}(\sigma) \;+\; 3\, I_{-1,1}(\sigma) \;-\; h N_1(\sigma)\;, 
\end{equation*}
where $N_1(\sigma)$ stands for the total number of spins equal to $+1$
in the configuration $\sigma$. To prove the assertion, it is therefore
enough to show that $h N_1(\sigma) \le I_{0,1}(\sigma) +
I_{-1,1}(\sigma)$.
	
By \cite[Assertion 4.A]{ll}, $I_{0,1}(\sigma) + I_{-1,1}(\sigma) \ge 4
\sqrt{N_1(\sigma)}$. It remains to obtain that $h N_1(\sigma) \le 4
\sqrt{N_1(\sigma)}$, i.e., that $h \sqrt{N_1(\sigma)} \le 4 $. Indeed,
since $N(\s) \leq n_0(n_0+1) +1$, $N_1(\sigma) \le n_0(n_0+1)+1$ so
that, by definition of $n_0$, $h \sqrt{N_1(\sigma)} \le h (n_0+1) \le
2+h \le 3$.
\end{proof}

Recall that $A(\sigma) = \{x\in\Lambda_L : \sigma_x \not = -1\}$. A
set $A\subset A(\sigma)$ is said to be a connected component of
$A(\sigma)$ if (a) for any $x$, $y\in A$, there exists a path $(x_0=x,
x_1, \dots, x_m=y)$ such that $x_i\in A$, $\Vert x_{i+1}-x_i\Vert =
1$, $0\le i <m$ and (b) for any $x\in A$, $y\not \in A$, such a path
does not exists.

Next assertion gives an estimation of the energy of a configuration
$\s \in \Omega_L$ such that $N(\s) \leq n_0 (n_0 +1) +1$ in terms of
the number of connected components.

\begin{asser}
\label{compo}
Let $\s \in \Omega_L$ be a configuration such that $N(\s) = n_0 (n_0
+1) $, and denote by $k$, $1 \leq k \leq n_0 (n_0 +1)$ the number of
connected components of $\s$. Then,
\begin{equation*}
\HH(\s) \;\geq\; \Gamma_c \;+\; 2\, (k-1)  \;+\; h \;.
\end{equation*}
\end{asser}

\begin{proof}
By Assertion \ref{fasB}, we can assume that $\s$ has no spin equal to
$+1$. For such a configuration and by definition of $\Gamma_c$,
\begin{equation*}
\HH(\s) \;=\; \big[\, I_{-1,0} (\s) \,-\, (4n_0 +4) \, \big] 
\;+\; \Gamma_c \;+\; h \;.
\end{equation*}
To complete the proof of the assertion, we have to show that $I_{-1,0}
(\s) \geq (4n_0 +4) + 2(k-1)$.
	
By moving 2 of the connected components of $\s$, and gluing them
together, we reach a new configuration $\s^1$ such that $N(\s^1) =
N(\s)$, while the size of the interface has decreased at least by 2:
\begin{equation*}
I_{-1,0} (\s) \;\geq\; I_{-1,0} (\s^1) \;+\; 2\; .
\end{equation*}
Iterating this argument $k-1$ times, we finally reach a configuration
$\s^*$ with only one connected component and such that
\begin{equation}
\label{8-1}
I_{-1,0} (\s) \;\geq\; I_{-1,0} (\s^*) \;+\; 2 (k-1)\; .
\end{equation}
The last connected component is glued to the set formed by the
previous ones in such a way that the set $A(\s^*)$ is not a $n_0
\times (n_0 +1)$ rectangle. This is always possible.

Since the connected set $A(\s^*)$ is not a $n_0 \times (n_0 +1)$
rectangle, by \cite[Assertion 4.B]{bl2},
$I_{-1,0} (\s^*) \geq 4 n_0 +4$, so that
\begin{equation*}
I_{-1,0} (\s) \geq (4n_0 +4) + 2(k-1),
\end{equation*}
which proves the assertion.
\end{proof}

We estimate the sum over $\s \in B_1$ on the right-hand side of
\eqref{ff3} in the next assertion.

\begin{asser}
\label{fasC}
There exists a finite constant $C_0$ such that
\begin{equation*}
\sum_{\s \in B_1} \sum_{\s' \sim \s} \mu_{\ms V}(\s) \land \mu_{\ms
  V}(\s') \;\le\;  C_0 \, \frac{1}{Z_{\ms V}} \, |\Lambda_L| \, e^{-\b
  \Gamma_c} \, e^{- h \b} 
\end{equation*}
for all $\beta \ge C_0$.
\end{asser}

\begin{proof}	
The proof of this assertion is divided in three steps. The first one
consists in applying Assertion \ref{fasB} to restrict the first sum to
configurations with no spin equal to $+1$. Indeed, as configuration in
$B_1$ have at most $n_0 (n_0+1)$ spins different from $-1$, by this
assertion, 
\begin{equation*}
\sum_{\s \in B_1} \sum_{\s' \sim \s} \mu_{\ms V}(\s) \land \mu_{\ms
  V}(\s') \;\le\; 2^{n_0(n_0+1)}
\sum_{\s \in B_{1,0}} \sum_{\s' \sim \s} \mu_{\ms V}(\s) \land \mu_{\ms
  V}(\s') \;, 
\end{equation*}
where $B_{1,0}$ represents the set of configurations in
$\{-1,0\}^{\Lambda}$ which belong to $B_1$. 

The second step consists in characterizing all configurations
$\sigma'$ which may appear in the second sum.  Recall that it is
performed over configurations $\sigma'\in \ms V_\moins \setminus B$
which can be obtained from $\sigma$ by one flip. In particular,
$N(\sigma')$, introduced in \eqref{ff1}, can differ from $N(\sigma) =
n_0(n_0+1)$ by at most by $1$.

If $N(\sigma')=N(\sigma)+1$, as $\sigma'\not\ni B \supset \mf R^+$, we
have that $\sigma'\in \mf R^a$. Since $\sigma$ does not belong to $\mf
R$, the $0$-spins of the configuration $\sigma$ form a $n_0\times (n_0
+1)$-rectangle in which one site has been removed and one site at the
boundary of the rectangle has been added. In this case $\bb H(\sigma)
> \bb H(\sigma')$ and $\bb H(\sigma) \ge \Gamma_c + h$. The last bound
is attained if the site removed from the rectangle to form $\sigma$ is
a corner. Hence, restricting the second sum to configurations
$\sigma'$ such that $N(\sigma')=N(\sigma)+1$ yields that
\begin{equation*}
\sum_{\s \in B_{1,0}} \sum_{\s' \sim \s} \mu_{\ms V}(\s) \land \mu_{\ms
  V}(\s') \;\le\;  C_0\, \frac{1}{Z_{\ms V}} \,
|\Lambda_L|  \, e^{-\b \Gamma_c} \, e^{- h \b} \;,
\end{equation*}
where the factor $|\Lambda_L|$ takes into account the possible
positions of the rectangle and the constant $C_0$ the positions of the
erased and added sites.

If $N(\sigma')=N(\sigma)$, resp. $N(\sigma')=N(\sigma)-1$, as $\sigma$
has no spin equal to $+1$, this means that one $0$-spin has been
flipped to $+1$, resp. to $-1$. In both cases, there are $n_0 (n_0
+1)$ such configurations $\sigma'$. Hence, the sum restricted to
such configurations $\sigma'$ is less than or equal to
\begin{equation*}
C_0 \sum_{\s \in B_{1,0}} \mu_{\ms V}(\s) \;.
\end{equation*}

Let $N := n_0 (n_0+1)$, and denote by $\ms C_k$, $1\le k \le N$, the
set of configurations in $B_{1,0}$ which have $k$ connected
components. Rewrite the previous sum according to the number of
components and apply Assertion \ref{compo} to obtain that it is
bounded by
\begin{equation*}
\sum_{k=1}^N \sum_{\s \in \ms C_k} \mu_{\ms V}(\s)  \;\le\;
C_0 \frac{1}{Z_{\ms V}} \, |\Lambda_L| \, e^{-\b \Gamma_c} \, e^{-h \b} 
\sum_{k=1}^N |\Lambda_L|^{k-1} \, e^{-\b[2(k-1)]}\;, 
\end{equation*}
where $|\Lambda_L|^k$ takes into account the number of positions of
the $k$ components, and $C_0$ the form of each component. By the
assumption of the theorems, $|\Lambda_L| \, e^{-2\b}$ is bounded by
$1/2$ for $\beta$ large enough, so that the sum is bounded by $2$. To
complete the proof of the assertion it remains to recollect the
previous estimates.
\end{proof}

\begin{proof}[Proof of Lemma \ref{fl1}]
This Lemma is a consequence of Assertions \ref{fasA}, \ref{fasC}, and
from the fact that $h<5$.
\end{proof}

\section{The growth of a supercritical droplet}
\label{proofs1} 

In the previous section we have seen that starting from $\moins$ we
hit the boundary of the valley $\ms V_\moins$ at $\mf R^a$. In this
section we show that starting from $\mf R^a$ the process either
returns to $\moins$, if the extra $0$-spin is attached to one of the
shortest sides of the rectangle, or it invades the all space with
positive probability, if the extra $0$-spin is attached to one of the
longest sides of the rectangle.

Denote by $\mf R^{l}$, $\mf R^{s}$ the configurations of $\mf R^a$ in
which the extra particle is attached to one of the longest, shortest
sides of the rectangle, respectively, and by $\mf R^{c}$ the
configurations of $\mf R^a$ in which the extra particle is attached to
one corner of the rectangle. Let $\mf R^{i} = \mf R^a \setminus \mf
R^c$, $\mf R^{lc} = \mf R^{l} \cap \ms R^{c}$, $\mf R^{li}= \mf R^{l}
\cap \mf R^{i}$.

Recall that $\ms M = \{\moins, \zero, \plus\}$, and let
\begin{equation}
\label{m3}
\delta (\beta) \:=\; 
|\Lambda_L|^{1/2}\, e^{-[(n_0+1)h-2]\beta} \;+\;  |\Lambda_L|^{1/2}\,  e^{-h\beta}
\;+\; |\Lambda_L|^{2} \, e^{-(2-h) \beta} \;.
\end{equation}

\begin{proposition}
\label{ml7}
There exists a finite constant $C_0$ such that for all $\sigma\in \mf
R^{lc}$, $\sigma'\in \mf R^{li}$, and $\sigma''\in \mf R^{s}$,
\begin{gather*}
\big| \, \bb P_{\sigma} [H_{\bf -1} = H_{\ms M}]  \,-\, 1/2\, \big|
\;\le\; C_0\, \delta (\beta)
\quad \text{and}\quad
\big| \, \bb P_{\sigma} [H_{\bf 0} = H_{\ms M}]  \,-\, 1/2\, \big|
\;\le\; C_0\, \delta (\beta)\;,  \\
\big| \, \bb P_{\sigma'} [H_{\bf -1} = H_{\ms M}]  \,-\, 1/3\, \big|
\;\le\; C_0\, \delta (\beta)
\quad \text{and}\quad
\big| \, \bb P_{\sigma'} [H_{\bf 0} = H_{\ms M}]  \,-\, 2/3\, \big|
\;\le\; C_0\, \delta (\beta)\;, \\
\bb P_{\sigma''} [H_{\bf -1} = H_{\ms M}]  \;\ge \; 1 \,-\, C_0\,
\delta (\beta) 
\end{gather*}
for all $\beta\ge C_0$.
\end{proposition}

The proof of this proposition is divided in several lemmata.  The
first result describes what happens when there is a $0$-spin attached
to the side of a rectangle of $0$-spins in a sea of $-1$-spins.  From
such a configuration, either the attached $0$-spin is flipped to $-1$
or an extra $0$-spin is created at the neighborood of the attached
$0$-spin.

For $n\ge 1$, let 
\begin{equation}
\label{m01}
\kappa_n (\beta) \;:=\; e^{-h \b} \;+\; n\, e^{-(2-h)\b} \;+\; 
|\Lambda_L|\, e^{-(4-h)\b}\;.
\end{equation}
Since $n^2\le |\Lambda_L|$,
$\kappa_n (\beta)\le \delta_1 (\beta)$, where
\begin{equation*}
\delta_1 (\beta) \;:=\; e^{-h \b} \;+\; |\Lambda_L|^{1/2} \, e^{-(2-h)\b}  \;+\; 
|\Lambda_L|\, e^{- (4-h)\b}\;.
\end{equation*}
Note that for $\beta$ large enough, $|\Lambda_L|\, e^{- (4-h)\b} \le
|\Lambda_L|^{1/2} \, e^{-(2-h)\b}$. 

\begin{asser}
\label{mas1}
Fix $n_0\le m \le n \le L-3$. Consider a configuration $\sigma$ with
$nm +1$ $0$-spins, all the other ones being $-1$. The $0$-spins form a
$(n\times m)$-rectangle and the extra $0$-spin has one neighbor
$0$-spin which sits at one corner of the rectangle {\rm [}there is
only one $-1$-spin with two $0$-spins as neighbors {\rm ]}. Let $\s_-$,
resp. $\s_+$, be the configuration obtained from $\s$ by flipping to
$-1$ the attached $0$-spin, resp. by flipping to $0$ the unique $-1$
spin with two $0$-spins as neighbors. Then, there exists a constant
$C_0$ such that
\begin{equation*}
\Big|\, p_\b (\s,\s_-) \,-\, \frac 12 \, \Big| 
\;\le\; C_0\, \delta_1 (\beta) \;, \quad 
\Big|\, p_\b (\s,\s_+) \,-\, \frac 12 \, \Big| 
\;\le\; C_0\, \delta_1 (\beta) \;.
\end{equation*}
\end{asser}

\begin{proof}
We prove the lemma for $\s_-$, the proof for $\s_+$ being
identical. Clearly, $R_\b (\s,\s_-) = R_\b (\s,\s_+) = 1$, so that 
\begin{equation*}
p_\b (\s,\s_-) \;=\; \frac{R_\b (\s,\s_-)}{\l_\b (\s)} \;=\;
\frac{1}{2 + \sum_{\s' \neq \s_-,\s_+} R_\b (\s,\s')}\;\cdot
\end{equation*}
Consider the last sum. There are three terms, corresponding to the
corners of the rectangle, for which $R_\b (\s,\s') \le e^{-\b h}$.
There are $4(n+m)-2$ terms, corresponding to the inner and outer
boundaries of the rectangle, such that $R_\b (\s,\s') \le e^{-\b
  (2-h)}$. All the remaining rates are bounded by $e^{-\b
  (4-h)}$. Hence,
\begin{equation*}
\sum_{\s' \neq \s_-,\s_+} R_\b (\s,\s') \;\le\;
C_0 \; \kappa_n(\beta) \;,
\end{equation*}
where $\kappa_n(\beta)$ has been introduced in \eqref{m01}.  This
proves the assertion.
\end{proof}

In the next assertion we consider the case in which the extra $0$-spin
does not sit at the corner of the rectangle, but in its interior. The
proof of this result, as well as the one of the next assertion, is
similar to the previous proof.

\begin{asser}
\label{mas2}
Under the same hypotheses of the previous assertion, assume now that
the extra $0$-spin has one neighbor $0$-spin which does not sit at one
corner of the rectangle {\rm [}there are exactly two $-1$-spins with
two $0$-spins as neighbors {\rm ]}. Let $\s_-$, resp. $\s^+_+$,
$\s^-_+$, be the configuration obtained from $\s$ by flipping to $-1$
the attached $0$-spin, resp. by flipping to $0$ one of the two
$-1$-spins with two $0$-spins as neighbors. Then, there exists a
constant $C_0$ such that
\begin{equation*}
\Big|\, p_\b (\s,\s_-) \,-\, \frac 13 \, \Big| 
\;\le\; C_0\, \delta_1 (\beta) \;, \quad 
\Big|\, p_\b (\s,\s^\pm_+) \,-\, \frac 13 \, \Big| 
\;\le\; C_0\, \delta_1 (\beta) \;.
\end{equation*}
\end{asser}

The next lemma states that once there are two adjacent $0$-spins
attached to one of the sides of the rectangle, this additional
rectangle increases with very high probability. This result will
permit to enlarge a $(p\times 1)$-rectangle to a $(2n_0 \times
1)$-rectangle. To enlarge it further we will apply Lemma \ref{ml2}
below.

This result will be used in three different situations:
\begin{itemize}
\item[(A1)] To increase in any direction a rectangle with $2$ adjacent
  $0$-spins whose distance from the corners is larger than $2n_0$ to a
  rectangle with $2n_0$ adjacent $0$-spins;
\item[(A2)] To increase in the direction of the corner a rectangle
  with $k\ge 2$ adjacent $0$-spins which is at distance $n_0$ or less
  than from one of the corners to a rectangle with adjacent $0$-spins
  which goes up to the corner; 
\item[(A3)] To increase a rectangle with $k<2n_0$ adjacent $0$-spins
  which includes one of the corners to a rectangle with $2n_0$
  adjacent $0$-spins.
\end{itemize}

Fix $n_0\le m \le n \le L-3$. Denote by $\sigma$ a configuration in
which $nm$ $0$-spins form a $(m\times n)$-rectangle in a sea of
$-1$-spins. Recall that we denote this rectangle by $A(\sigma)$, and
assume, without loss of generality, that $m$ is the length and $n$ the
height of $A(\sigma)$. Let $(x,y)$ be the position of the upper-left
corner of $A(\sigma)$.

We attach to one of the sides of $A(\sigma)$ and extra $(p\times
1)$-rectangle of $0$-spins, where $p \ge 2$. To fix ideas, suppose
that the extra $0$-spins are attached to the upper side of length $m$
of the rectangle. 

More precisely, denote by $\eta^{(c,d)}$, $0\le c<d \le m$, $d-c\ge
2$, the configuration obtained from $\sigma$ by flipping from $-1$ to
$0$ the $([d-c]\times 1)$-rectangle, denoted by $R_{c,d}$, given by
$\{(x+c,y+1), \dots, (x+d,y+1)\}$. The next lemma asserts that before
anything else happens the rectangle $R_{c,d}$ increases at least by
$n_0$ units at each side.

For a pair $(c,d)$ as above, denote by $\ms S_{c,d}$ the set of
configurations given by
\begin{equation*}
\ms S_{c,d} \;=\; \{\eta^{(a,b)} : 0\le a \le c  \text{ and } d\le b\le
m\}\;,
\end{equation*}
and by $E_{c,d}$ the exit time from $\ms S_{c,d}$,
\begin{equation*}
E_{c,d} \;=\; \inf \big\{ t>0 : \sigma_t \not\in \ms S_{c,d} \big\}\;.
\end{equation*}
Let $c^* = \max \{0, c-n_0\}$, $d^* = \min \{m, d+n_0\}$.  Denote
by $H_{c,d}$ the hitting time of the set $\ms S_{c^*,d^*}$:
\begin{equation*}
H_{c,d} \;=\; \inf \big\{ t>0 : \sigma_t \in \ms S_{c^*,d^*} \big\}\;.
\end{equation*}

\begin{lemma}
\label{mas3}
There exists a constant $C_0$ such that
\begin{equation*}
\bb P_{\eta^{(c,d)}} \big[ E_{c,d} < H_{c,d} \big]\;\le \; 
C_0\, \delta_1 (\beta) \;. 
\end{equation*}
\end{lemma}

\begin{proof}
Consider a configuration $\eta^{(a,b)}$ in $\ms S_{c,d}$. To fix ideas
assume that $a>0$, $b<m$. At rate $1$ the $-1$-spins at $(x+a-1,
y+1)$, $(x+b+1, y+1)$ flip to $0$. Consider all other possible spin
flips. There are less than $2\, |\Lambda_L|$ flips whose rates are
bounded by $e^{-[4-h]\beta}$, $4(n+m)$ flips whose rates are bounded
by $e^{-[2-h]\beta}$ and $4$ flips whose rates are bounded by $e^{- h
  \beta}$. Since all these jumps are independent, the probability that
the $-1$-spin at $(x+a-1, y+1)$ flips to $0$ before anything else
happens is bounded below by $ 1 \,-\, C_0\, [\, |\Lambda_L| \,
e^{-[4-h]\beta} + n e^{-[2-h]\beta} + e^{- h \beta} \, ]$. Iterating
this argument $n_0$-times yields the lemma.
\end{proof}

Applying Assertion \ref{mas1} or \ref{mas2} and then Lemma \ref{mas3}
to a configuration $\sigma\in \mf R^a$ yields that either the process
returns to $\mf R$ or an extra row or line of $0$-spins is added to
the rectangle of $0$-spins. The next two lemmata describe how the
process evolves after reaching such a configuration.

Denote by $m\le n$ the length of the rectangle of $0$-spins.  If the
shortest side has length $n_0$ or less, the configuration evolves to a
$(m\times [n-1])$ rectangle of of $0$-spins. If both sides are
supercritical, that is if $m>n_0$, a $-1$-spin next to the rectangle
is flipped to $0$.

Denote by $\cs_L$ the set of stable configurations of $\Omega_L$,
i.e., the ones which are local minima of the energy:
\begin{equation*}  
\label{defstable}
\cs_L \;=\; \big\lbrace \s \in \Omega_L: \bb H(\sigma) < \bb H(\sigma^{x,\pm})
\text{ for all } x\in \Lambda_L \big\rbrace \;. 
\end{equation*}
Let
\begin{equation*}
\delta_2 (\beta) \;=\;  |\Lambda_L|\, e^{-[4 - n_0 h]\beta } \,+\,
e^{- h \beta}\;.
\end{equation*}

Fix $2\le m\le n_0$, $2 \le n \le n_0+1$, $m\le n$. Consider a
configuration $\sigma$ with $nm$ $0$-spins forming a $(n\times
m)$-rectangle, all the other ones being $-1$. If $m=n=2$, let $\ms
S(\sigma) = \{\moins\}$. If this is not the case, let $\ms S(\sigma)$
be the pair (quaternion if $m=n$) of configurations in which a row or
a column of $0$-spins of length $m$ is removed from the rectangle
$A(\sigma)$. 

We define the valley of $\sigma$, denoted by $\ms V_{\sigma}$, as
follows. Let $\ms G_k$, $0\le k \le m$, be the configurations which
can be obtained from $\sigma$ by flipping to $-1$ a total of $k$
$0$-spins surrounded, at the moment they are switched, by two
$-1$-spins. In particular, the elements of $\ms G_1$ are the four
configurations obtained by flipping to $-1$ a corner of $A(\sigma)$.

Let $\ms G = \cup_{0\le k <m} \ms G_k$. Note that we do not include
$\ms G_m$ in this union. Let $\ms B$ be the configurations which do
not belong to $\ms G$, but which can be obtained from a configuration
in $\ms G$ by flipping one spin.  Clearly, $\ms S(\sigma)$ and $\ms
G_m$ are contained in $\ms B$. Finally, let $\ms V(\sigma) = \ms G
\cup \ms B$ be the neighborhood of $\sigma$.

\begin{lemma}
\label{ml8}
Fix $2\le m\le n_0$, $2 \le n \le n_0+1$, $m\le n$. Consider a
configuration $\sigma$ with $nm$ $0$-spins forming a $(n\times
m)$-rectangle, all the other ones being $-1$. Then, there exists a
constant $C_0$ such that
\begin{equation*}
\bb P_\sigma \big[\, H_{\ms B \setminus \ms S(\sigma)} <  H_{\ms
  S(\sigma)}\,\big]  \;\le \; C_0\, \delta_2 (\beta) \;.
\end{equation*}
\end{lemma}

\begin{proof}
Assume that $2<m<n$. The other cases are treated in a similar way. 
As in the proof of Proposition \ref{Lemma1}, denote by $\eta_t$ the
process $\sigma_t$ reflected at $\ms V_\sigma$, and by $\bb
P_{\sigma}^{\ms V}$ its distribution starting from $\sigma$. By 
\eqref{probatriangle}, 
\begin{equation*}
\bb P_\sigma \big[\, H_{\ms B \setminus \ms S(\sigma)} <  H_{\ms
  S(\sigma)}\,\big] \;=\; 
\bb P_\sigma ^{\ms V} \big[\, H_{\ms B \setminus \ms S(\sigma)} <  H_{\ms
  S(\sigma)}\,\big] \;\le\; \frac{\capa_{\ms V}(\sigma, \ms B \setminus
  \ms S(\sigma))}
{\capa_{\ms V}(\sigma, \ms S(\sigma))}\;,
\end{equation*}
where $\capa_{\ms V}$ represents the capacity with respect to the
process $\eta_t$.

We estimate separately these two capacities. Let $\eta^{(k)}$, $0\le
k\le m$, be a sequence of configurations such that $\eta^{(0)} =
\sigma$, $\eta^{(m)} \in \ms S(\sigma)$, and $\eta^{(k+1)}$ is
obtained from $\eta^{(k)}$ by flipping to $-1$ a $0$-spin surrounded
by two $-1$-spins.

Consider the flow $\varphi$ from $\sigma$ to $\ms S(\sigma)$ given by
$\varphi (\eta^{(k)}, \eta^{(k+1)})=1$ and $\varphi=0$ for all the other
bonds. By Thomson's principle,
\begin{equation*}
\frac{1}
{\capa_{\ms V}(\sigma, \ms S(\sigma))} \;\le\; m\, Z_{\ms V}\, e^{[\bb
  H(\sigma) + (m-1)h] \beta}\;.
\end{equation*}

To estimate the capacity on the numerator, denote by $\chi = \chi_{\ms
  B \setminus \ms S(\sigma)}$ the indicator function of the set $\ms B
\setminus \ms S(\sigma)$. By the Dirichlet principle,
\begin{equation*}
\capa_{\ms V}(\sigma, \ms B \setminus \ms S(\sigma))  \;\le\;
D_{\ms V} (\chi) \;\le\; \sum_{\sigma' \in \ms B \setminus \ms
  S(\sigma)} \sum_{\sigma''} \mu_{\ms V}(\sigma')\, \wedge \mu_{\ms
  V}(\sigma'') \;,
\end{equation*}
where the last sum is performed over all configurations $\sigma'' \in 
\ms V_\sigma \setminus [\ms B \setminus \ms S(\sigma)]$ which can be
obtained from $\sigma'$ by one flip. 

To estimate the last sum we examine all elements of
$\ms B \setminus \ms S(\sigma)$. There are at most $C_0\, |\Lambda_L|$
configurations $\sigma'$ obtained from a configuration in $\ms G$ by
flipping a spin at distance $2$ or more from the [inner or outer]
boundary of $A(\sigma)$. These configurations have only one neighbor
$\sigma''$ in $\ms V_\sigma$ and their energy is bounded below by
$\bb H(\sigma) + 4 - h$.

There are at most $C_0$ configurations $\sigma'$ not in $\ms G_m$ and
obtained from a configuration in $\ms G$ by flipping a spin
[surrounded by three spins of the same type] at the boundary of
$A(\sigma)$. These configurations have only one neighbor $\sigma''$ in
$\ms V_\sigma$ and their energy is bounded below by $\bb H(\sigma) + 2 - h$.

Finally there are at most $C_0$ configurations $\sigma'$ in
$\ms G_m\setminus \ms S(\sigma)$ or obtained from a configuration in
$\ms G$ by flipping a spin [surrounded by two spins of the same type]
at the boundary of $A(\sigma)$. These configurations have at most
$C_0$ neighbors $\sigma''$ in $\ms V_\sigma$ and their energy is
bounded below by $\bb H(\sigma) + mh$. It follows from the previous
estimates that
\begin{equation*}
D_{\ms V} (\chi) \;\le\; C_0\, \frac{1}{Z_{\ms V}} e^{- \bb H(\sigma) \beta}
\big\{\, |\Lambda_L|\, e^{-[4 - h]\beta } \,+\, e^{- mh \beta}\,\big\}\;. 
\end{equation*}

Putting together the previous estimates on the capacity, we conclude
that
\begin{equation*}
\bb P_\sigma \big[\, H_{\ms B \setminus \ms S(\sigma)} <  H_{\ms
  S(\sigma)}\,\big] \;\le\; C_0\, 
\big\{\, |\Lambda_L|\, e^{-[4 - n_0 h]\beta } \,+\, e^{- h \beta} \}\;.
\end{equation*}
This completes the proof of the lemma.
\end{proof}

Applying the previous result repeatedly yields that starting from a
configuration $\sigma$ with $nm$ $0$-spins forming a $(n\times
m)$-rectangle in a sea of $-1$-spins the process converges to $\moins$
if the shortest side has length $m\le n_0$. 

\begin{corollary}
\label{ml9}
Let $\sigma$ be a configuration with $n_0(n_0+1)$ $0$-spins which form
a $n_0\times (n_0+1)$-rectangle in a background of $-1$. Then,
\begin{equation*}
\bb P_{\sigma} [H_{\bf -1} = H_{\ms M}] \;\ge\; 1 \;-\; C_0 \, \delta_2(\beta)\;.
\end{equation*}
\end{corollary}

The next results shows that, in constrast, if $m>n_0$, then the
rectangle augments.  We first characterize how the process leaves the
neighborhood of such a configuration $\sigma$.

Fix $n_0< m \le n \le L-3$. Consider a configuration $\sigma$ with
$nm$ $0$-spins forming a $(n\times m)$-rectangle in a sea of
$-1$'s. Recall that we denote by $A(\sigma)$ the rectangle of
$0$-spins. Let $\ms V_\sigma$ be the valley of $\sigma$ whose elements
can be constructed from $\sigma$ as follows.

Fix $0\le k\le n_0$. We first flip sequentially $k$ spins of
$A(\sigma)$ from $0$ to $-1$. At each step we only flip a $0$-spin if
it is surrounded by two $-1$-spins. The set of all configurations
obtained by such a sequence of $k$ flips is represented by $\ms
G_k$. In particular, since at the beginning we may only flip the
corners of $A(\sigma)$, $\ms G_1$ is composed of the four
configurations obtained by flipping to $-1$ one corner of
$A(\sigma)$. On the other hand, since $m>n_0$, all configurations of
$\ms G_k$ have an energy equal to $\bb H(\sigma) + kh$. Denote by $\ms
G_{-1}$ the configuration obtained from $\sigma$ by flipping to $0$ a
$-1$-spin which is surrounded by one $0$-spin. Let $\ms G =
\cup_{-1\le k<n_0} \ms G_k$, and note that $\ms G_{n_0}$ has not been
included in the union.

The second and final stage in the construction of the valley $\ms
V_\sigma$ consists in flipping a spin of a configuration in $\ms
G$. More precisely, denote by $\ms B$ all configurations which are not
in $\ms G$, but which can be obtained from a configuration in $\ms G$
by flipping one spin. The set $\ms B$ is interpreted as the boundary
of the valley $\ms V_\sigma := \ms G \cup \ms B$.

Note that all configurations in $\ms V_\sigma$ can be obtained from
$\sigma$ by at most $n_0$ flips. Conversely, if $(\eta^{(0)} = \sigma,
\eta^{(1)}, \dots, \eta^{(n_0)})$ is a sequence of configurations
starting from $\sigma$ in which each element is obtained from the
previous one by flipping a different spin, one of the configurations
$\eta^{(k)}$ belong to the boundary of $\ms V_\sigma$.
 
Denote by $\ms R_{2}$ the set of $2(m+n-2)$ configurations obtained
from $\s$ by flipping to $0$ two adjacent $-1$-spins, each of which is
surrounded by a $0$-spin. Clearly, $\ms R_{2}$ is contained in $\ms
B$, and the energy of a configuration in $\ms R_{2}$ is equal to
$\bb H_{-1} = \bb H_{0} - h$, where $\bb H_{0} := \bb H(\sigma) +
(2-h)$ is the configuration in which only one $-1$-spin has flipped to
$0$. As $n_0 h > 2-h$, an inspection shows that all the elements of
$\ms A := \ms B \setminus \ms R_{2}$ have an energy strictly
larger that $\bb H_0$. In particular, starting from $\sigma$, the
process reaches the boundary $\ms B$ at $\ms R_{2}$. This is
the content of the next lemma.

Let 
\begin{equation}
\label{m02}
\delta_3 (\beta) \;=\; e^{-[(n_0+1)h-2]\beta} \;+\; 
|\Lambda_L|^{1/2} \, e^{-(2-h)\beta} \;+\; |\Lambda_L|\, e^{-2\beta}\;.
\end{equation}

\begin{lemma}
\label{ml1}
Fix $n_0< m \le n \le L-3$. Consider a configuration $\sigma$ with
$nm$ $0$-spins forming a $(n\times m)$-rectangle in a sea of $-1$'s.
Recall that $\ms A := \ms B \setminus \ms R_{2}$. 
Then, there exists a constant $C_0$ such that
\begin{equation*}
\PP_\s [ H_{\ms A} < H_{\ms R_{2}} ]  \;\le \; C_0\, \delta_3 (\beta) \;.
\end{equation*}
\end{lemma}

\begin{proof}
Since we may not leave the set $\ms V_\sigma$ without crossing its
boundary $\ms B$, the probability appearing in the statement of
the lemma is equal to the one for the reflected process at $\ms
V_\sigma$, that is, the one in which we forbid jumps
from $\ms V_\sigma$ to its complement. We estimate the probability
for this later dynamics which is restricted to $\ms V_\sigma$.

By \eqref{probatriangle}, the probability appearing in the statement
of the lemma is bounded above by $\capa_{\ms V} (\s , \ms A)/
\capa_{\ms V} (\s , \ms R_{2})$, where $\capa_{\ms V}$ stands for the
capacity with respect to the reflected process. We estimate the numerator
by the Dirichlet principle and the denominator by the Thomson
principle. 

We start with the denominator. Denote by $\eta^{(1)}, \dots,
\eta^{(2(n+m-2))}$, the configurations of $\ms R_2$, and by $\bs x_j$,
$\bs y_j \in \bb Z^2$ the positions of the two extra $0$-spins of
$\eta^{(j)}$. Assume that $\bs x_j \not = \bs x_k$ for $j\not =
k$. Consider the flow $\varphi$ from $\sigma$ to $\ms R_{2}$ such that
$\varphi(\sigma, \sigma^{\bs x_j}) = 1/[2(n+m-2)]$,
$\varphi(\sigma^{\bs x_j}, \eta^{(j)}) = 1/[2(n+m-2)]$, and $\varphi =
0$ at all the other bonds. By the Thomson principle, since $\mu_{\ms
  V}(\sigma^{\bs x_j})$ is less than or equal to $\mu_{\ms V}(\sigma)$
and $\mu_{\ms V}(\eta^{(j)})$,
\begin{equation}
\label{m2}
\frac 1{\capa_{\ms V} (\s , \ms R_{2})} \;\le\; \frac 1{n+m-2}\,
Z_{\ms V}\, e^{\beta \bb H_0}\;.
\end{equation}

We turn to the numerator.  Denote by $f$ the indicator function of the
set $\ms A$. Since $f$ vanishes at $\sigma$ and is equal to $1$
at $\ms A$, by the Dirichlet principle, $\capa_{\ms V} (\s , \ms
A)\le D_{\ms V}(f)$. On the other hand,
\begin{equation}
\label{m1}
D_{\ms V}(f) \;=\; \sum_{\eta\in \ms A} \sum_{\xi \sim \eta} \mu_{\ms
  V}(\eta) \wedge \mu_{\ms V}(\xi)\;,
\end{equation}
where the second sum is performed over all configurations in
$\ms V_\sigma \setminus A$ which can be obtained from $\eta$
by one spin flip. This relation is represented by $\xi\sim \eta$.

We first consider the configuration $\eta$ in $\ms A$ which have a
neighbor in $\ms G_{-1}$. Fix $\xi\in \ms G_{-1}$. Consider the
configurations obtained from $\xi$ by flipping a spin which is not at
the boundary of $A(\sigma)$. There are at most $|\Lambda_L|$ of such
spins, and the energy of the configurations obtained by this spin flip
is bounded below by $\bb H_0 + 4 - h$. There is one special spin,
though, the one which is next to the extra spin and not at the
boundary of $A(\sigma)$. The energy of the configuration obtained by
flipping this spin to $0$ or to $+1$ is bounded below by $\bb H_0 + 2
- h$. The contribution of these terms to \eqref{m1} is thus bounded
above by
\begin{equation*}
2\, (n+m)\, \frac 1{Z_{\ms V}}\, e^{-\beta \bb H_0}\, \big\{\, |\Lambda_L|\, 
e^{-(4-h) \beta} \,+\, e^{-(2-h) \beta} \,\big\} \;,
\end{equation*}
where the factor $2\, (n+m)$ comes from the total number of
configurations in $\ms G_{-1}$.

We turn to the configurations obtained from $\xi$ by flipping a spin
at the boundary of $A(\sigma)$. Since the configuration resulting from
this flip can not be in $\ms R_2$, their energy is bounded below by
$\bb H_0 + 2 - h$. The contribution of these terms to the sum
\eqref{m1} is thus bounded by
\begin{equation*}
4\, (n+m)^2\, \frac 1{Z_{\ms V}}\, e^{-\beta \bb H_0}\, e^{-(2-h) \beta} \;,
\end{equation*}
the extra factor $2(n+m)$ coming from the possible positions of the
extra spin flip at the boundary.

Consider now configurations $\eta$ in $\ms A$ which have a neighbor in
a set $\ms G_{k}$, $0\le k<n_0$. Fix $0\le k<n_0$ and $\xi \in \ms
G_{k}$.  The configuration $\xi$ is formed by a connected set $A(\xi)
\subset A(\sigma)$ of $0$-spins in a sea of $-1$-spins.

There is one special case which is examined separately. Suppose that
$\xi$ belongs to $\ms G_{n_0-1}$ and $\eta$ to $\ms G_{n_0}$. There
are $C(n_0)$ of such pairs, and the energy of $\eta$ is equal
to $\bb H(\sigma) + n_0h = \bb H_0 + (n_0+1) h - 2$. We exclude from now
in the analysis these pairs.

Apart from this case, there are two types of configurations $\eta\in
\ms A$ which can be obtained from $\xi$ by a spin flip. The first ones
are the ones in which $\eta$ and $\xi$ differ by a spin which belongs
to the inner or outer boundary of $A(\xi)$. There are at most $4(n+m)
\le 8n$ of such configurations. The energy of these configurations is
bounded below by $\bb H(\xi) + 2 - h = \bb H(\sigma) + kh + 2 - h = \bb
H_0 + kh$. The minimal case occurs when a $-1$-spin which has a
$0$-spin as neighbor is switched to $0$.

The previous estimate is not good enough in the case $k=0$ because in
the argument we did not exclude the configurations in $\ms
G_{-1}$. For $k=0$ if $\eta$ belongs to $\ms B \setminus \ms
G_{-1}$, we obtain that $\bb H(\eta) \ge \bb H(\sigma) + 2 + h = \bb
H_0 + 2h$. The right-hand side of this inequality corresponds to the
case in which a $0$-spin surrounded by three $0$-spins has been changed
to $-1$.  In conclusion, if the flip occurs at the boundary of
$A(\xi)$, there are at most $8n$ configurations and the energy of such
a configuration is bounded below by $\bb H_0 + h$. 

If the flip did not occur at the boundary of $A(\xi)$, there are at
most $|\Lambda_L|$ possible configurations, and the energy of these
configurations is bounded below by $\bb H(\xi) + 4 - h = \bb
H(\sigma) + kh + 4 - h = \bb H_0 + 2 + kh$. 

The previous estimates yield that the Dirichlet form \eqref{m1} is
bounded by
\begin{equation*}
\frac{C_0}{Z_{\ms V}} e^{-\beta \bb H_0} \Big\{
n\, |\Lambda_L|\, e^{-(4-h)\beta} \,+\, n^2 \, e^{-(2-h)\beta} \,+\, 
e^{-[(n_0+1)h-2]\beta} + n e^{-\beta h} + |\Lambda_L|
e^{-2\beta}\Big\}\;. 
\end{equation*}
Multiplying this expression by \eqref{m2} yields that the probability
appearing in the statement of the lemma is bounded above by
\begin{equation*}
C_0\, \Big\{ \, n \, e^{-(2-h)\beta} \,+\, 
e^{-[(n_0+1)h-2]\beta} \,+\, |\Lambda_L|\, e^{-2\beta} \Big\}
\end{equation*}
because $(n_0+1)h-2 < h$ and $4-h>2$. We bounded $1/n$ by $1$ when $n$
appeared in the denominator because $n$ can be as small as $n_0+1$.
This completes the proof of the lemma since $n^2 \le |\Lambda_L|$. 
\end{proof}

The previous lemma asserts that the process leaves the neighborhood of
a large rectangle of $0$-spins in a sea of $-1$ spins by switching
from $-1$ to $0$ two adjacent spins at the outer boundary of the
rectangle. At this point, applying Lemma \ref{mas3} yields that with a
probability close to $1$ these two adjacent $0$-spins will increase to
$2n_0$ adjacent $0$-spins. To increase it further, we apply the next
lemma.

This result will be used in two different situations: 
\begin{itemize}
\item[(B1)] To increase in any direction a rectangle with $2n_0$
  adjacent $0$-spins whose distance from the corners is larger than
  $2n_0$ to a rectangle of adjacent $0$-spins which is at distance
  less than $2n_0$ from one of the corners;
\item[(B2)] To increase a rectangle with $k\ge 2n_0$ adjacent
  $0$-spins which contains one corner and is at a distance larger than
  $2n_0$ from the other corner to a rectangle of adjacent $0$-spins
  which is at distance less than $2n_0$ from this later corner.
\end{itemize}

To avoid a too strong assumption on the rate at which the cube
$\Lambda_L$ increases, we do not impose [as in the Assertions
\ref{mas1}--\ref{mas2} and Lemma \ref{mas3}] the extra rectangle of
$0$-spins to grow without never shrinking or to grow while the spins
at the corners stay put.

As in the proof of Lemma \ref{ml1}, we construct a set of
configurations in two stages. We consider below the case in which the
extra rectangle is far from the corners. The case in which it contains
one of the corners can be handled similarly.

Fix $n_0\le m \le n \le L-3$. Denote by $\sigma$ a configuration in
which $nm$ $0$-spins form a $(n\times m)$-rectangle in a sea of
$-1$-spins. Denote this rectangle by $A(\sigma)$, and assume, without
loss of generality, that $m$ is the length and $n$ the height of
$A(\sigma)$. Let $(x,y)$ be the position of the upper-left corner of
$A(\sigma)$.

We attach to one of the sides of $A(\sigma)$ an extra
$(p\times 1)$-rectangle of $0$-spins, where $p > n_0$. In particular,
the length of the side to which this extra rectangle is attached has
to be larger than $n_0$. To fix ideas, suppose that the extra
$0$-spins are attached to the upper side of length $m$ of the
rectangle and assume that $m>5n_0$. As explained previously, the case
$m\le 5n_0$ is handled by Lemma \ref{mas3}.

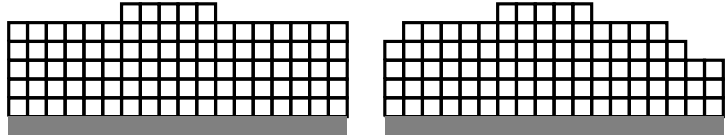
\begin{figure}[!h]
\centering
\begin{tikzpicture}[scale = .25]
\foreach \y in {0, ..., 4}
\foreach \x in {0, ..., 17}
\draw[very thick] (\x, \y) -- (\x +1, \y) -- (\x
+ 1, \y + 1) -- (\x , \y + 1) -- (\x , \y );
\foreach \y in {5}
\foreach \x in {6, ..., 10}
\draw[very thick] (\x, \y) -- (\x +1, \y) -- (\x
+ 1, \y + 1) -- (\x , \y + 1) -- (\x , \y );
\fill[gray] (0, -1) -- (18, -1) -- (18, 0) -- (0,0) -- 
 	(0,-1);	
\foreach \y in {0, ..., 2}
\foreach \x in {20, ..., 37}
\draw[very thick] (\x, \y) -- (\x +1, \y) -- (\x
+ 1, \y + 1) -- (\x , \y + 1) -- (\x , \y );
\foreach \y in {3}
\foreach \x in {20, ..., 35}
\draw[very thick] (\x, \y) -- (\x +1, \y) -- (\x
+ 1, \y + 1) -- (\x , \y + 1) -- (\x , \y );
\foreach \y in {4}
\foreach \x in {21, ..., 34}
\draw[very thick] (\x, \y) -- (\x +1, \y) -- (\x
+ 1, \y + 1) -- (\x , \y + 1) -- (\x , \y );
\foreach \y in {5}
\foreach \x in {26, ..., 30}
\draw[very thick] (\x, \y) -- (\x +1, \y) -- (\x
+ 1, \y + 1) -- (\x , \y + 1) -- (\x , \y );
\fill[gray] (20, -1) -- (38, -1) -- (38, 0) -- (20,0) -- 
 	(20,-1);	
	\end{tikzpicture}
	\caption{Assume that $n_0=3$. The first picture provides an
          example of a configuration $\eta^{(c,d)}$. Here,
          $m=18\le n$, $c=6$, $d=10$ and $p=5$. The gray portion
          indicates that the rectangle continues below as its height
          is larger than $18$.  The second picture presents a
          configuration in $\ms G_{c,d,6}$. We chose $k=6>n_0$ to make
          the definition clear.}
	\label{fig4}
\end{figure}

Denote by $\eta^{(c,d)}$, $2n_0\le c<d \le m-2n_0$, $d-c>n_0$, the
configuration obtained from $\sigma$ by flipping from $-1$ to $0$ the
$([d-c]\times 1)$-rectangle, denoted by $R_{c,d}$, given by
$\{(x+c,y+1), \dots, (x+d,y+1)\}$. Denote by $\ms G_{c,d,k}$, $0\le
k\le n_0$, the configurations obtained from $\eta^{(c,d)}$ by
sequentially flipping to $-1$, close to the corners of $A(\sigma)$, a
total of $k$ $0$-spins surrounded, at the moment they are switched, by
two $-1$ spins. We do not flip spins in $R_{c,d}$.

In the case [not considered below] where the rectangle $R_{c,d}$
includes one corner, say $c=0$, we treat the spins at $(x,y+1), \dots,
(x+n_0,y+1)$ as belonging to the corner and we allow them to be
flipped.

Let $\ms G = \cup_{c,d} \cup_{0\le k <n_0} \ms G_{c,d,k}$, where the
first union is performed over all indices such that $2n_0\le c<d \le
m-2n_0$, $d-c>n_0$. Note that we excluded $k=n_0$ in this
union. Denote by $\ms B$ the configurations which do not belong to
$\ms G$ and which can be obtained from a configuration in $\ms G$ by
flipping one spin. The set $\ms B$ is treated as the boundary of $\ms
G$. 

Note that $\ms B$ contains configurations in $\ms G_{c,d,n_0}$ and
also configurations in $\ms G_{c,d,k}$ in which $d-c=n_0$. Let $\ms
A_1$, $\ms A_2$ be such configurations: 
\begin{equation*}
\ms A_1 \; :=\;  \bigcup_{c,d} \ms G_{c,d,n_0} \;, \quad 
\ms A_2 \; :=\; \bigcup_{c',d'} \bigcup_{0\le k <n_0} \ms G_{c',d',k}\;,
\quad \ms A \; :=\; \ms A_1 \cup \ms A_2\;,
\end{equation*}
where the first union is performed over all indices such that $2n_0 \le
c<d \le m-2n_0$, $d-c>n_0$, and the second one is performed over all
indices such that $2n_0 < c'<d' \le m-2n_0$, $d'-c'=n_0$. The set $\ms
B$ also contains configurations in which a $0$-spin in a rectangle
$R_{c,d}$ surrounded by $3$ $0$-spins is flipped to $\pm 1$.

All configurations in $\ms G$ are similar to the ones represented in
Figure \ref{fig4}. They are obtained by adding a
$(p\times 1)$-rectangle of $0$-spins to the upper side of $A(\sigma)$
and by switching to $-1$ some of the spins of $A(\sigma)$ close to the
corners.

For $t< H_{\ms B}$, denote by $c_t$, resp. $d_t$, the position at time
$t$ of the leftmost, resp. rightmost, $0$-spin of the upper
rectangle.  Let $\tau^*$ be the first time $c_t \le 2n_0$ or
$d_t \ge m-2n_0$:
\begin{equation*}
\tau^* \;:=\; \inf\{t\ge 0:  c_t \le 2n_0 \text{ or } d_t \ge
m-2n_0\}\;. 
\end{equation*}
and let $\delta_4'(\beta) = |\Lambda_L| \, e^{-[4-h]\beta} \;+\;
|\Lambda_L|^{1/2} \, e^{-[2-h]\beta}$,
\begin{equation}
\label{m04}
\delta_4(\beta) \;:=\; |\Lambda_L|^{3/2} \, e^{-[4-h]\beta} 
\;+\; |\Lambda_L| \, e^{-[2-h]\beta}\;.
\end{equation}

\begin{lemma}
\label{ml2}
Let $\sigma' = \eta^{(c_0,d_0)}$, for some $2n_0 < c_0<d_0 < m-2n_0$,
$d_0-c_0\ge 2n_0$. Then, there exists a finite constant $C_0$ such
that 
\begin{equation*}
\PP_{\s'} [ H_{\ms B} < \tau^* ]  \;\le \; C_0\, \delta_4 (\beta) 
\end{equation*}
for all $\beta\ge C_0$.
\end{lemma}

\begin{proof}
Let $C_t$ be the set of spins in $A(\sigma)$ close to the corners which
takes the value $-1$ at time $t$. In the right picture of Figure
\ref{fig4} the set $C_t$ consists of the $6$ squares at the corners
which have been removed from the left picture. Set $C_t$ to be $\Lambda_L$ for
$t\ge H_{\ms B}$. Before hitting $\ms B$, the total number of sites of
$C_t$, represented by $|C_t|$, is strictly bounded by $n_0$. Moreover,
Before hitting $\ms B$, $|C_t|$, which starts from $0$, is bounded by
a Markov process $m_t$ which jumps from $k\ge 0$ to $k+1$ at rate $n_0
e^{-\beta h}$ and from $k+1$ to $k$ at rate $1$.

Let $b_t = \mb 1\{\sigma_t \in \ms G^c\setminus \ms A\}$. This process
starts from $0$ and jumps to $1$ when $\sigma_t$ reaches $\ms B$
through a configuration which is not in $\ms A$. Inspecting all
possible jumps yields that the process $b_t$ is bounded by a process
$z_t$ which starts from $0$ and jumps to $1$ at rate $|\Lambda_L|
e^{-[4-h]\beta} + 2(n+m) e^{-[2-h]\beta} \le \delta_4'(\beta)$, where
$\delta_4'(\beta)$ has been introduced just above \eqref{m04}

The key observation in the proof of this lemma is that the processes
$(c_t,d_t)$, $C_t$ and $b_t$ are independent until the set $\ms B$ is
attained because they involve different spin jumps.

Let $x_t = d_t - c_t$. Before hitting $\ms B$, $x_t$ evolves as a
random walk in $\bb Z$ which starts from $d-c\ge 2n_0$ and jumps from
$k$ to $k+1$ at rate $2$ and from $k+1$ to $k$ at rate $2 e^{-\beta
  h}$. Let $\tau_0$ be the first time $x_t\le n_0$.

Let $H^b_1$ be the hitting time of $1$ by the process $b_t$, and let
$H^C_{n_0}$ be the first time $|C_t|$ attains $n_0$.  The event
$\{H_{\ms B} < \tau^*\}$ is contained in the event $\{\tau_0 <
\tau^*\} \cup \{H^b_1 < \tau^*\} \cup \{H^C_{n_0} < \tau^*\}$.

Consider three independent Markov chains, $X_t$, $Y_t$, $Z_t$.  The
first one takes value in $\{0, \dots, m\}$, it starts from $2n_0$, and
jumps from $k$ to $k+1$ at rate $2$ and from $k+1$ to $k$ at rate $2
e^{-\beta h}$.  The process $Y_t$ takes value in $\{0, \dots, n_0\}$,
it starts from $0$, and jumps from $k$ to $k+1$ at rate $2 e^{-\beta
  h}$ and from $k+1$ to $k$ at rate $1$. The last one takes value in
$\{0, 1\}$, it starts from $0$, and jumps from $0$ to $1$ at rate
$\delta_4'(\beta)$.

Before time $H_{\ms B}$ we may couple $(x_t, |C_t|, b_t)$ with $(X_t,
Y_t, Z_t)$ in such a way that $X_t = x_t$, $|C_t| \le Y_t$ and $b_t\le
Z_t$. In particular, $\{\tau_0 < \tau^*\} \subset \{H^X_0 <
H^X_{m}\}$, $\{H^C_{n_0} < \tau^*\} \subset \{H^Y_{n_0} < H^X_{m}\}$,
$\{H^b_1 < \tau^*\} \subset \{H^Z_1 < H^X_{m}\}$. In these formulas,
$H^W$ stands for the hitting time of the process $W$. Hence, the
probability appearing in the statement of the lemma is bounded above
by
\begin{equation}
\label{m03}
P\big[ H^X_{n_0} < H^X_{m} \big] \;+\; P\big[ H^Y_{n_0} < H^X_{m} < H^X_{n_0}\big] 
\;+\; P\big[ H^Z_1 < H^X_{m} < H^X_{n_0} \big] \;.
\end{equation}
We estimate each term separately.

The first one is easy. Denote by $P^X_k$ the distribution of $X_t$
starting from  $k$. Let $f(k) = P^X_k[H^X_{n_0} < H^X_{m}]$, which is
harmonic. It can be computed explicitly and one gets that
\begin{equation*}
P\big[ H^X_{n_0} < H^X_{m} \big] \;=\; P^X_{2n_0}\big[ H^X_{n_0} < H^X_{m} \big]
\;\le\; 2 e^{-n_0 h \beta}
\end{equation*}
provided $\beta\ge C_0$.

We turn to the second term of \eqref{m03}. On the set $\{H^X_{m} <
H^X_{n_0}\}$ we may replace $X$ by a random walk on $\bb Z$ and
estimate $P\big[ H^Y_{n_0} < H^X_{m}]$. As $X$ and $Y$ are independent,
we condition on $Y$ and treat $H^Y_{n_0}$ as a positive real number.
The set $\{H^Y_{n_0} < H^X_{m}\}$ is contained in $\{ X_{J} \le m\}$
where $J=H^Y_{n_0}$. Fix $\theta>0$. By the exponential Chebyshev
inequality and since $m\le n$ [the sizes of the rectangle $A(\sigma)$], 
\begin{equation*}
P^X_{2n_0}\big[ X_{J} \le m \big] \;\le\; P^X_{0}\big[ X_{J} \le n
\big] \;\le\; e^{\theta n} E^X_{0}\big[ e^{-\theta X_{J}} \big]\;.
\end{equation*}
Choose $\theta = 1/n$ and compute the expectation to obtain that the
previous expression is bounded by $3 \, \exp\{ -(2/n) H^Y_{n_0}\}$
provided $\beta\ge C_0$. By Lemma \ref{ml3}, 
\begin{equation*}
E \big[ e^{ -(2/n) H^Y_{n_0}} \big] \;\le\; C_0 \, n\, e^{- n_0 h
  \beta}\;,
\end{equation*}
where $E$ represents the expectation with respect to $P$.

The third expression in \eqref{m03} is estimated similarly. The
argument yields that it is bounded by 
\begin{equation*}
3 \, E \big[ e^{ -(2/n) H^Z_{1}} \big] \;\le\;
3\, \frac{\delta_4'(\beta)}{(2/n) + \delta_4'(\beta)}
\;\le\; 2\, n\, \delta_4'(\beta)  \;,
\end{equation*}
where $\delta_4'(\beta)$ has been introduced just above \eqref{m04}. This
completes the proof of the lemma because $n_0h > 2-h$.
\end{proof}

\begin{remark}
\label{rm2}
One could improve the previous argument and obtain a better estimate
by allowing the spins at the boundary of $A(\sigma)$ to flip while the
rectangle $R_{c,d}$ fills the upper side.
\end{remark}

The next result describes how the supercritical droplet of $0$-spins
grows. Let
\begin{equation*}
\delta_5 (\beta) \;:=\; e^{-[(n_0+1)h-2]\beta} \;+\;  e^{-h\beta}
\;+\; |\Lambda_L|^{3/2} \, e^{-(2-h) \beta} \;.
\end{equation*}
A simple computation based on the bound $|\Lambda_L| e^{-2\beta}\le
1$, which holds for $\beta$ large enough, shows that there exists
$C_0$ such that
\begin{equation}
\label{m4}
\delta_1 (\beta)  \;+\; \delta_3 (\beta) \;+\; 
|\Lambda_L|^{1/2} \, \delta_4 (\beta) \;\le\; C_0 \,\delta_5 (\beta)
\end{equation}
for all $\beta \ge C_0$.

\begin{proposition}
\label{ml11}
Fix $n_0<m\le n\le L$.  Let $\sigma$ be a configuration with $mn$
$0$-spins which form a $m\times n$-rectangle in a background of
$-1$-spins. Then, there exists a constant $C_0$ such that
\begin{equation*}
\bb P_{\sigma} [H_{\ms S_L\setminus\{\sigma\}} = H_{\ms S(\sigma)}]
\;\ge\; 1 \,-\, C_0\, \delta_5 (\beta) \;,
\end{equation*}
for all $\beta\ge C_0$. In this equation, if $n\le L-3$, $\ms
S(\sigma)$ is the set of four configurations in which a row or a
column of $0$ spins is added to the rectangle $A(\sigma)$. If $m <
n=L-2$, the set $\ms S(\sigma)$ is a triple which includes a band of
$0$ spins of width $m$ and two configurations in which a row or a
column of $0$ spins of length $n$ is added to the rectangle
$A(\sigma)$.  If $m\le L-3$, $n=L$, the set $\ms S(\sigma)$ is a pair
formed by two bands of $0$ spins of width $m+1$. If $m= n=L-2$, $\ms
S(\sigma)$ is a pair of two bands of width $L-2$. If $n_0<m=L-2$,
$n=L$, $\ms S(\sigma) =\{{\bf 0}\}$.
\end{proposition}

\begin{proof}
Consider the first case, the proof of the other ones being similar.
By Lemma \ref{ml1}, with a probability close to $1$, the process
$\sigma_t$ escapes from the valley $\ms V_\sigma$ of $\sigma$ by
flipping to $0$ two adjacent spins at the outer boundary of
$A(\sigma)$. By Lemma \ref{mas3}, with a
probability close to $1$, these two adjacent spins will become $2n_0$
adjacent $0$-spins. Of course, if the length of the side is smaller
than $2n_0$, this simply means that the $0$-spins fill the side.

Denote by $R_e$ the $(2n_0 \times 1)$-rectangle of adjacent $0$-spins.
At this point, if $R_e$ is at distance less than $2n_0$ of one of the
corners of $A(\sigma)$, we apply Lemma \ref{mas3}
again to extend it up to the corner. After this step, or if $R_e$ is
at distance greater than $2n_0$ of one of the corners of $A(\sigma)$,
we apply Lemma \ref{ml2} to increase $R_e$ up to the point that one of
its extremities is at a distance less than $2n_0$ of one of the
corners of $A(\sigma)$. We fill the $2n_0$ sites with $0$-spins by
applying again Lemma \ref{mas3}. We repeat the
procedure applying  Lemma \ref{ml2} to reach a position close to the
corner and then Lemma \ref{mas3} to fill the gap.

The probability that something goes wrong in the way is bounded by the
sum of the probabilities that each step goes wrong. This is given by
$C_0 \{\delta_3 (\beta) + 6n_0 \delta_1 (\beta) + 2 |\Lambda_L|^{1/2}
\delta_4 (\beta)\}$, which completes the proof of the proposition in
view of \eqref{m4}.
\end{proof}

\begin{corollary}
\label{ml10} 
Let $\sigma$ be a configuration with $n_0(n_0+1)+2$ $0$-spins which
form a $n_0\times (n_0+1)$-rectangle in a background of $-1$, with two
additional adjacent $0$-spins attached to the longest side of the
rectangle. Then, there exists a constant $C_0$ such that
\begin{equation*}
\bb P_{\sigma} [H_{\zero} = H_{\ms M}] \;\ge\; 1 \;-\; C_0 \, 
|\Lambda_L|^{1/2}\, \delta_5(\beta)
\end{equation*}
for all $\beta\ge C_0$.
\end{corollary}

\begin{proof}
Denote by $\{T_j : j\ge 1\}$ the jump times of the process
$\sigma_t$. Let $\sigma_+$ be the configuration obtained from $\sigma$
by flipping to $0$ all $-1$-spins in the same row or column of the two
adjacent $0$-spins. Hence, $\sigma_+$ has $(n_0+1)^2$ $0$-spins in a
background of $-1$ spins. According to Lemma \ref{mas3},
$\bb P_{\sigma} [ \sigma(T_{n_0-1}) = \sigma_+] \ge 1 - C_0\,
\delta_1(\beta)$. Since $T_{n_0-1} \le H_{\ms M}$, we may apply the
strong Markov property to conclude that
\begin{equation*}
\bb P_{\sigma} [H_{\plus} = H_{\ms M}] \;\ge\; \bb P_{\sigma_+}
[H_{\plus} = H_{\ms M}] \,-\,  C_0\, \delta_1(\beta) \;.
\end{equation*}
At this point, apply Proposition \ref{ml11} $(2 |\Lambda_L|^{1/2})$
times to complete the proof.
\end{proof}

\begin{proof}[Proof of Proposition \ref{ml7}]
We prove the first statement, the argument for the
other ones being analogous. Fix $\sigma\in \mf R^{lc}$. Recall that we
denote by $T_1$ the time of the first jump. By the strong Markov
property at time $T_1$ and by Assertion \ref{mas1},
\begin{equation*}
\begin{aligned}
& \bb P_{\sigma} [H_{\bf -1} = H_{\ms M}] \;=\; 
\bb E_{\sigma} \big[ \, \bb P_{\sigma_{T_1}} [H_{\bf -1} = H_{\ms M}]
\, \big] \\
&\quad =\; \frac 12\, \big\{ \bb P_{\sigma_+} [H_{\bf -1} = H_{\ms M}]
\;+\;  \bb P_{\sigma_-} [H_{\bf -1} = H_{\ms M}] \, \big\} \;+\; R_\beta\;,  
\end{aligned}
\end{equation*}
where $R_\beta$ is a remainder whose absolute values is bounded by $
C_0\, \delta_1 (\beta)$, and $\s_-$, resp. $\s_+$, is the
configuration obtained from $\s$ by flipping to $-1$ the attached
$0$-spin, resp. by flipping to $0$ the unique $-1$ spin with two
$0$-spins as neighbors.

The configuration $\s_-$ has $n_0(n_0+1)$ $0$-spins which form a
$n_0\times (n_0+1)$-rectangle in a background of $-1$. Hence, by
Corollary \ref{ml9},
$\bb P_{\sigma_-} [H_{\bf -1} = H_{\ms M}] \ge 1 - C_0\, \delta_2
(\beta)$.
On the other hand, by Corollary \ref{ml10},
$\bb P_{\sigma_+} [H_{\bf +1} = H_{\ms M}] \ge 1 - C_0\,
|\Lambda_L|^{1/2}\, \delta_5 (\beta)$.
The first statement of the proposition follows from these estimates
and from the fact that $\delta_2(\beta) < \delta_5(\beta)$, because
$4-n_0h>2$, and $\delta_3(\beta) < C_0\, \delta_5(\beta)$ for $\beta$
large enough by \eqref{m4}.
\end{proof}

\section{Proof of Theorems \ref{mt22} and \ref{mainprop}}
\label{sec3}

The proofs of Theorems \ref{mainprop} and \ref{mt22} are based on
Propositions \ref{Lemma1} and \ref{ml7}.  By Proposition \ref{ml7},
there exists a finite constant $C_0$ such that
\begin{equation}
\label{m5}
\begin{aligned}
& \max_{\sigma\in \mf R^a}\, 
\bb P_{\sigma} [H_{\bf +1}  < H_{\{{\bf -1} ,\bf{0}\}}]  \;\le \;
C_0\, \delta (\beta)\;, \\
&\quad \max_{\sigma\in \mf R^l}\, 
\bb P_ {\sigma} [H_{\bf -1} < H_{\{{\bf 0},{\bf +1}\}} ]  \;\le \;
(1/2) \;+\;  C_0\, \delta (\beta)
\end{aligned}
\end{equation}
for all $\beta\ge C_0$, where $\delta(\beta)$ has been introduced in
\eqref{m3}. 

\begin{proof}[Proof of Theorem \ref{mainprop}]
We prove the first statement of the theorem, the argument for the
second one being identical.
Recall the definition of the boundary $\mf B^+$ of the valley of
$\moins$ introduced in \eqref{eq09}. By \eqref{ff02} and by the strong
Markov property at time $H_{\mf B^+}$,
\begin{equation*}
\bb P_{\moins} \big[ H_{\plus} < H_\zero \big] \;=\;
\bb E_{\moins} \big[ \, \bb P_{\sigma (H_{\mf B^+})} \big[ H_{\plus} < H_\zero \big]\,
\big]\; .
\end{equation*}
Let $q(\sigma) = \bb P_ {\bf -1} [H_{\sigma} = H_{\mf B^+} ]$,
$\sigma\in\mf B^+$. By Proposition \ref{Lemma1}, the previous
expectation is equal to
\begin{equation*}
\sum_{\sigma\in\mf B^+}
q(\sigma) \, \bb P_ {\sigma} [H_{\bf +1} < H_{\bf 0} ]
\;\le \; \sum_{\sigma\in\mf R^a}
q(\sigma) \, \bb P_ {\sigma} [H_{\bf +1} < H_{\bf 0} ] \;+\; 
C_0\, \epsilon(\beta) \;,
\end{equation*}
where $\epsilon(\beta)$ has been introduced in \eqref{m6}.

By the first estimate in \eqref{m5}, uniformly in $\sigma\in\mf R^a$, 
\begin{align*}
\bb P_ {\sigma} [H_{\bf +1} < H_{\bf 0} ] \; &\le\;
\bb P_ {\sigma} [H_{\bf +1} < H_{\bf 0} \,,\,
H_{\{{\bf -1},{\bf 0}\}} < H_{\bf +1} ]
\;+\; C_0\, \delta(\beta) \\
&=\; \bb P_ {\sigma} [ H_{\bf -1} < H_{\bf +1} < H_{\bf 0}]
\;+\; C_0\, \delta(\beta)\;.
\end{align*}
Therefore, by the strong Markov property at time $H_{\bf -1}$, 
\begin{equation*}
\bb P_ {\bf -1} [H_{\bf +1} < H_{\bf 0} ]\;\le\;
\bb P_ {\bf -1} [H_{\bf +1} < H_{\bf 0} ] 
\sum_{\sigma\in\mf R^a} q(\sigma) \, \bb P_ {\sigma} [H_{\bf -1} <
H_{\{{\bf 0},{\bf +1}\}} ]\, \;+\; C_0 \, \delta (\beta)  \;.
\end{equation*}
because $\epsilon(\beta) \le \delta (\beta)$.  

By the second bound in \eqref{m5}, as $\delta (\beta) \to 0$, for
$\sigma\in \mf R^l$, $\bb P_ {\sigma} [H_{\bf -1} < H_{\{{\bf 0},{\bf
    +1}\}} ] \le 2/3$ provided $\beta \ge C_0$. Hence, for $\beta$
large enough,
\begin{align*}
\bb P_ {\bf -1} [H_{\bf +1} < H_{\bf 0} ]\; & \le\;
(2/3)\, \bb P_ {\bf -1} [H_{\bf +1} < H_{\bf 0} ] 
\sum_{\sigma\in\mf R^a} q(\sigma) \;+\; C_0 \, \delta (\beta)  \\
& \le\; (2/3)\, \bb P_ {\bf -1} [H_{\bf +1} < H_{\bf 0} ] 
\;+\; C_0 \, \delta (\beta) \;.
\end{align*}
This completes the proof of the theorem since $\delta (\beta) \to 0$.
\end{proof}

\begin{proof}[Proof of Theorem \ref{mt22}]
Since the chains hits $\mf B^+$ before reaching $\bf 0$ and $\mf R^l$,
by the strong Markov property,
\begin{equation*}
\bb P_{\bf -1} [ H_{\mf R^l} < H_{\bf 0}] \;=\;
\sum_{\sigma\in\mf B^+} \bb P_{\bf -1} [ H_{\sigma} = H_{\mf B^+} ] \,
\bb P_{\sigma} [ H_{\mf R^l} < H_{\bf 0}]\;. 
\end{equation*}

Recall the definition of $q(\sigma)$ introduced in the previous proof.
By Proposition \ref{Lemma1}, this expression is equal to
\begin{equation}
\label{m7}
\sum_{\sigma\in\mf R^l} q(\sigma) \;+\;
\sum_{\sigma\in\mf R^s} q(\sigma) \, \bb P_{\sigma} [ H_{\mf R^l} <
H_{\bf 0}] \;+\; R (\beta) \;. 
\end{equation} 
where the absolute value of the remainder $R (\beta)$ is bounded by
$C_0\, \epsilon (\beta)$.  By Assertion \ref{mas1}, Lemma \ref{mas3}
and by the proof Lemma \ref{ml8}, uniformly in $\sigma\in\mf R^s$,
$\sigma'\in\mf R$
\begin{align*}
& \bb P_{\sigma} [ H_{\mf R} < H_{\mf R^l
  \cup\{{\bf -1}, {\bf 0}\}}]\;\ge \; 1 \,-\, C_0\, [\delta_1(\beta) \,+\,
\delta_2(\beta)]  \;, \\
&\quad \quad \bb P_{\sigma'} [ H_{\bf -1} < H_{\mf R^l
  \cup\{ {\bf 0}\}}]\;=\; 1 \,-\, C_0\, \delta_2(\beta) \;. 
\end{align*} 
Hence, uniformly in $\sigma\in\mf R^s$,
\begin{equation}
\label{m8}
\bb P_{\sigma} [ H_{\bf -1} < H_{\mf R^l 
\cup\{ {\bf 0}\}}]\;\ge \; 1 \,-\, C_0\, 
[\delta_1(\beta) \,+\, \delta_2(\beta)]  \;, 
\end{equation} 
and we may introduce the set $\{H_{\bf -1} < H_{\mf R^l \cup\{ {\bf
    0}\}}\}$ inside the probability appearing in \eqref{m7} by paying
a cost bounded by $C_0 [\delta_1(\beta) \,+\, \delta_2(\beta)]$.

Up to this point, we proved that
\begin{equation*}
\bb P_{\bf -1} [ H_{\mf R^l} < H_{\bf 0}] \;=\; 
\sum_{\sigma\in\mf R^l} q(\sigma) \;+\;
\sum_{\sigma\in\mf R^s} q(\sigma) \, \bb P_{\sigma} [ H_{\bf -1} <  H_{\mf R^l} <
H_{\bf 0}] \;+\; R (\beta) \;,
\end{equation*}
where the absolute value of the remainder $R (\beta)$ is bounded by
$C_0\, [\epsilon (\beta) \,+\, \delta_1(\beta) \,+\, \delta_2(\beta)]$.
By the strong Markov property this expression is equal to
\begin{equation*}
\sum_{\sigma\in\mf R^l} q(\sigma) \;+\; \bb P_{\bf -1} [ H_{\mf R^l} <
H_{\bf 0}]\, \sum_{\sigma\in\mf R^s} q(\sigma) \, \bb P_{\sigma} [
H_{\bf -1} <  H_{\mf R^l \cup \{\bf 0\}}] \;+\; R (\beta) \;.
\end{equation*}
By \eqref{m8}, this expression is equal to
\begin{equation*}
\sum_{\sigma\in\mf R^l} q(\sigma) \;+\; \bb P_{\bf -1} [ H_{\mf R^l} <
H_{\bf 0}]\, \sum_{\sigma\in\mf R^s} q(\sigma) \;+\; R (\beta) \;,
\end{equation*}
where the value of $R (\beta)$ has changed but not its
bound. Therefore, 
\begin{equation*}
\Big(\, 1\,-\, \sum_{\sigma\in\mf R^s} q(\sigma) \,\Big) \,
\bb P_{\bf -1} [ H_{\mf R^l} < H_{\bf 0}] \;=\; 
\sum_{\sigma\in\mf R^l} q(\sigma) \;+\; R (\beta) \;.
\end{equation*}

Since, by Proposition \ref{Lemma1},
\begin{equation*}
\sum_{\sigma \in \mf R^l \cup \mf R^s} q(\sigma) \;=\;  \bb P_{\bf -1} [
H_{\mf R^a} = H_{\mf B^+}] \;\ge \; 1\,-\, \epsilon(\beta)\;,
\end{equation*}
replacing on the right-hand side $\sum_{\sigma\in\mf R^l} q(\sigma)$
by $ 1 - \sum_{\sigma\in\mf R^s} q(\sigma) - R'(\beta)$, where the
absolute value of $R'(\beta)$ is bounded by $\epsilon(\beta)$, we
conclude that
\begin{equation*}
\bb P_{\bf -1} [ H_{\mf R^l} < H_{\bf 0}] \;=\; 
1 \;+\; R (\beta) \;,
\end{equation*}
as claimed.
\end{proof}

\section{The convergence of the trace process}
\label{sec5}

In this section, we examine the evolution of the trace of $\sigma_t$
on $\ms M = \{{\bf -1}, {\bf 0}, {\bf +1}\}$ under the hypotheses of
Theorem \ref{mainprop}.  Denote by $\eta_t$ the trace of $\sigma_t$ on
$\ms M$. We refer to Section \ref{sec2} for a precise definition. By
\cite[Proposition 6.1]{bl2}, $\eta_t$ is an $\ms M$-valued,
continuous-time Markov chain. Recall the definition of $\theta_\beta$
given in \eqref{71}.

\begin{proposition}
\label{p71}
As $\beta\uparrow\infty$, the speeded-up Markov chain $\eta(\theta_\beta t)$
converges to the continuous-time Markov chain on $\ms M$ in which $\bf
+1$ is an absorbing state, and whose jump rates $\mb r(\eta, \xi)$, are
given by
\begin{equation*}
\mb r({\bf -1}, {\bf 0}) \;=\; \mb  r({\bf 0},{\bf +1})  \;=\; 1 \;,
\quad \mb r({\bf -1},{\bf +1}) \;=\; \mb r({\bf 0},{\bf -1})\;=\;
0\;. 
\end{equation*}
\end{proposition}

The proof of this proposition, presented at the end of this section,
relies on estimation of capacities. We start characterizing the
distribution of $\sigma(H_{\mf B^+})$ when the process starts from
$\moins$.  Recall the definition of $\delta_2(\beta)$ introduced just
before Lemma \ref{ml8}.

\begin{lemma}
\label{as61}
There exists a finite constant $C_0$ such that
for every $\sigma\in \mf R^a$, 
\begin{equation*}
\Big|\, |\mf R^a|\, \bb P_ {\bf -1}
[H_{\sigma} = H_{\mf B^+} ] \,-\, 1\, \Big| 
\;\le \; C_0\, [\, \epsilon(\beta) + \delta_1(\beta)\,]
\end{equation*}
for all $\beta\ge C_0$.
\end{lemma}

\begin{proof}
Fix a reference configuration $\sigma^*$ in
$\mf R^a$. By \eqref{probastar} and by definition of the
capacity, 
\begin{equation*}
\bb P_ {\bf -1} [H_{\sigma} = H_{\mf B^+} ] \;=\; \frac{M({\bf -1}) 
\, \bb P_ {\bf -1} [H_{\sigma} = H^+_{\mf B^+ \cup\{{\bf -1}\}} ]}
{\capa ({\bf -1}, \mf B^+)}\;\cdot
\end{equation*}
By reversibility, the numerator of this expression is equal to
\begin{equation*}
M(\sigma) \, \bb P_ {\sigma} [H_{\bf -1} = H^+_{\mf B^+ \cup\{{\bf -1}\}} ]
\;=\; \mu_\beta(\sigma) \, \lambda_\beta(\sigma) \, 
\bb P_ {\sigma} [H_{\bf -1} = H^+_{\mf B^+ \cup\{{\bf -1}\}} ]\;.
\end{equation*}
By Assertions \ref{mas1} and \ref{mas2}, with a probability close to
$1$, either a negative spin next to the attached $0$-spin flips to $0$
or the attached $0$-spin flips to $-1$. In the first case, as the
process left the valley $\ms V_\moins$ introduced at the beginning of
Section \ref{energy}, it will hit $\mf B^+$ before reaching $\moins$. In
the second case, applying Lemma \ref{ml8} repeatedly yields that the
process reaches $\moins$ before hitting $\mf B^+$. Hence, by these
three results,
\begin{equation*}
\big|\, \bb P_ {\sigma} [H_{\bf -1} = H^+_{\mf B^+ \cup\{{\bf -1}\}} ]
\,-\, \mf n (\sigma) \,\big| \;\le\; C_0 \, [\, \delta_1(\beta) \,+\,
\delta_2(\beta)\,]\;,
\end{equation*}
where 
\begin{equation*}
\mf n(\sigma) \;=\;
\begin{cases}
1/2  & \text{ if $\sigma\in\mf R^{c}$,} \\
1/3 & \text{ if $\sigma\in\mf R^{i}$.}
\end{cases}
\end{equation*}
Since
\begin{equation*}
\lambda(\sigma) \;=\;
\begin{cases}
2 + \delta_1(\beta) & \text{ if $\sigma\in\mf R^{c}$,} \\
3 + \delta_1(\beta) & \text{ if $\sigma\in\mf R^{i}$,}
\end{cases}
\end{equation*}
and since $\mu_\beta(\sigma) = \mu_\beta(\sigma^*)$, we conclude that
\begin{equation*}
\bb P_ {\bf -1} [H_{\sigma} = H_{\mf B^+} ] \;=\; 
\frac{\mu_\beta(\sigma^*)}
{\capa ({\bf -1}, \mf B^+)} \, \big( 1 + R_\beta \big)\;,
\end{equation*}
where the absolute value of $R_\beta$ is bounded by
$C_0 [\delta_1(\beta) + \delta_2(\beta)]$.  Summing over
$\sigma\in\mf R^a$, it follows from Proposition \ref{Lemma1} that for
any configuration $\sigma^*\in\mf R^a$,
\begin{equation*}
\Big|\, \frac {\mu_\beta(\sigma^*)\, |\mf R^a|} {\capa ({\bf -1}, \mf B^+)}
\,-\, 1\,\Big| \;\le\; C_0 \, [\epsilon(\beta) + \delta_1(\beta)] 
\end{equation*}
because $\delta_2(\beta) \le \epsilon(\beta)$. To complete the proof
of the lemma, it remains to multiply both sides of the penultimate
identity by $|\mf R^a|$.
\end{proof}

It follows from the proof of the previous lemma and the identity $|\mf
R^a| = 4(2n_0+1) |\Lambda_L|$ that there exists a finite constant
$C_0$ such that for all $\sigma\in\mf R^a$,
\begin{equation}
\label{410}
\Big|\, \frac {\capa ({\bf -1}, \mf B^+)} {\mu_\beta(\sigma)\, |\Lambda_L|} 
\,-\, 4(2n_0+1) \,\Big| \;\le\; C_0 \, [\epsilon(\beta) + \delta_1(\beta)]
\end{equation}
for all $\beta\ge C_0$.

\begin{proposition}
\label{mt3}
For any configuration $\eta\in\mf R^l$ and any configuration
$\xi\in\mf R^l_0$, 
\begin{equation*}
\lim_{\beta\to\infty} \frac{\capa ({\bf -1},\{{\bf 0}, {\bf +1}\})}
{\mu_\beta (\eta)\, |\Lambda_L|}\;=\;\frac{4(2n_0+1)}3\, \;=\; 
\lim_{\beta\to\infty} \frac{\capa ({\bf 0},\{{\bf -1}, {\bf +1}\})}
{\mu_\beta (\xi)\, |\Lambda_L|}\;.
\end{equation*}
\end{proposition}

\begin{proof}
We prove below the first idendity of the proposition, the one of the
second being analogous. We first claim that
\begin{equation}
\label{601}
\capa ({\bf -1} , \{{\bf 0}, {\bf +1}\}) \;=\;
\capa ({\bf -1} , \mf B^+)
\sum_{\sigma\in\mf B^+} \bb P_{\bf -1} [ H_{\sigma} = H_{\mf B^+} ]  \,
\bb P_{\sigma} [ H_{\{{\bf 0}, {\bf +1}\} }  < H_{\bf -1} ] \;.
\end{equation}
Indeed, since starting from $\bf -1$ the process hits $\mf B^+$ before
$\{{\bf 0}, {\bf +1}\}$, by the strong Markov property we have that
\begin{equation*}
\bb P_{\bf -1} [ H_{\{{\bf 0}, {\bf +1}\} }  < H^+_{\bf -1} ] \;=\;
\sum_{\sigma\in\mf B^+} \bb P_{\bf -1} [ H_{\sigma}
= H^+_{\mf B^+ \cup \{\bf -1\}} ]  \,
\bb P_{\sigma} [ H_{\{{\bf 0}, {\bf +1}\} }  < H_{\bf -1} ] \;.
\end{equation*}
By \eqref{probastar}, we may rewrite the previous expression as
\begin{equation*}
\bb P_{\bf -1} [ H_{\mf B^+} < H^+_{\bf -1} ]  
\sum_{\sigma\in\mf B^+} \bb P_{\bf -1} [ H_{\sigma} = H_{\mf B^+} ]  \,
\bb P_{\sigma} [ H_{\{{\bf 0}, {\bf +1}\} }  < H_{\bf -1} ] \;.
\end{equation*}
This proves \eqref{601} in view of the definition \eqref{defcapa} of
the capacity. 

By \eqref{410} and \eqref{601}, for any configuration $\sigma^*\in\mf
R^a$,
\begin{equation*}
\frac{\capa ({\bf -1} , \{{\bf 0}, {\bf +1}\})} 
{|\Lambda_L|\, \mu_\beta(\sigma^*)} \;=\; [4(2n_0+1) + R^{(1)}_\beta] 
\sum_{\sigma\in\mf B^+} \bb P_{\bf -1} [ H_{\sigma} = H_{\mf B^+} ]  \,
\bb P_{\sigma} [ H_{\{{\bf 0}, {\bf +1}\} }  < H_{\bf -1} ] \;.
\end{equation*}
where $|R^{(1)}_\beta| \le C_0 \, [\epsilon(\beta) + \delta_1(\beta)]$
for $\beta$ large.  

Consider the sum. By Proposition \ref{Lemma1}, we may ignore the terms
$\sigma\not \in \mf R^a$. On the other hand, Proposition \ref{ml7}
provides the asymptotic value of $\bb P_{\sigma} [ H_{\{{\bf 0}, {\bf
    +1}\} }  < H_{\bf -1} ]$ for $\sigma\in \mf R^a$. Putting together
these two result yields that the sum is equal to
\begin{equation*}
\frac 12\, \bb P_{\bf -1} [ H_{\mf R^{lc}} = H_{\mf B^+} ]  \;+\;
\frac 23\, \bb P_{\sigma} [ H_{\mf R^{li}} = H_{\mf B^+} ] \;+\;
R^{(2)}_\beta \;,
\end{equation*}
where the absolute value of the remainder $R^{(2)}_\beta$ is bounded
by $C_0[\epsilon (\beta) + \delta(\beta)]$.  By Lemma \ref{as61}, this
expression is equal to
\begin{equation*}
\frac 12\,  \frac{2}{2n_0+1} \;+\; \frac 23\, \frac{n_0-1}{2n_0+1}
\;+\; R^{(3)}_\beta \;,
\end{equation*}
where $|R^{(3)}_\beta| \le C_0 \, [\epsilon (\beta) + \delta(\beta)]$
because $\delta_1(\beta) \le \delta(\beta)$. The first assertion of
the proposition follows from the previous estimates.
\end{proof}

The same proof yields that for any configuration $\eta\in\mf R^l$,
$\xi\in\mf R^l_0$,
\begin{equation*}
\lim_{\beta\to\infty} \frac{\capa ({\bf -1}, {\bf 0} )}
{\mu_\beta (\eta)\, |\Lambda_L|}\;=\;\frac{4(2n_0+1)}3\, \;=\; 
\lim_{\beta\to\infty} \frac{\capa ({\bf 0}, {\bf +1} )}
{\mu_\beta (\xi)\, |\Lambda_L|}\;.
\end{equation*}
In particular,
\begin{equation}
\label{m9}
\lim_{\beta\to\infty} 
\frac{\capa ({\bf -1}, {\bf 0})}{\capa ({\bf -1}, \{{\bf 0} , 
{\bf  +1}\})}\;=\; 1 \;, \quad
\lim_{\beta\to\infty} 
\frac{\capa ({\bf 0}, {\bf +1})}{\capa ({\bf 0}, \{{\bf -1} , 
{\bf  +1}\})}\;=\; 1 \;.
\end{equation}

\begin{corollary}
\label{as71b}
We have that
\begin{equation*}
\lim_{\beta\to\infty} 
\, \frac{\capa({\bf 0}, \ms M\setminus\{{\bf 0}\})}{\mu_\beta({\bf 0})} 
\, \theta_\beta\;=\; 1\;, 
\end{equation*}
where $\theta_\beta$ has been introduced in \eqref{71}.
\end{corollary}

\begin{proof}
Fix $\eta\in \mf R^a$ and $\xi\in \mf R^a_0$. By definition of
$\theta_\beta$, the expression appearing in the statement of the
corollary can be written as
\begin{equation*}
\frac{\mu_\beta({\xi})}{\mu_\beta({\bf 0})} \,
\frac{\capa({\bf 0}, \ms M\setminus\{{\bf 0}\})}
{\mu_\beta({\xi}) \, |\Lambda_L|} 
\, \frac {\mu_\beta({\eta})\, |\Lambda_L|} 
{\capa({\bf -1}, \ms M\setminus\{{\bf -1}\})}\,
\frac{\mu_\beta({\bf -1})}{\mu_\beta({\eta})}\;\cdot
\end{equation*}
By the previous lemma, the product of the second and third expression
converges to $1$, while the first and fourth term cancel.
\end{proof}

\begin{lemma}
\label{as64}
We have that
\begin{equation*}
\lim_{\beta\to\infty} 
\frac{\capa ({\bf -1}, {\bf +1})}{\capa ({\bf -1}, \{{\bf 0} , 
{\bf  +1}\})}\;=\; 1 \;.
\end{equation*}
\end{lemma}

\begin{proof}
Fix a configuration $\sigma^*$ in $\mf R^a$.
By the proof of Proposition \ref{mt3},
\begin{equation*}
\frac{\capa ({\bf -1}, {\bf +1})}{\mu_\beta(\sigma^*)\, |\Lambda_L|}\;=\; 
4(2n_0+1) \sum_{\sigma\in\mf R^a} \bb P_{\bf -1} [ H_{\sigma} = H_{\mf B^+} ]  \,
\bb P_{\sigma} [ H_{ {\bf +1}}  < H_{\bf -1} ] \;+\; R^{(1)}_\beta \;,
\end{equation*}
where $|R^{(1)}_\beta| \le C_0 [\epsilon(\beta) + \delta(\beta)]$ for
some finite constant $C_0$.

By Proposition \ref{ml7}, starting from $\sigma\in \mf R^a$ we reach
$\{\moins, \zero\}$ before $\plus$ with a probability close to
$1$. Hence, up to a small error, we may include the event $ H_{\{{\bf
    -1},\zero\}} < H_{\bf +1}$ inside the second probability which
becomes $\{H_\zero < H_{ {\bf +1}} < H_{\bf -1}\}$. Applying the
strong Markov property, the right-hand side of the previous expression
becomes
\begin{equation*}
4(2n_0+1) \, \bb P_{\zero} [ H_{ {\bf +1}}  < H_{\bf -1} ] 
\sum_{\sigma\in\mf R^a} \bb P_{\bf -1} [ H_{\sigma} = H_{\mf B^+} ]  \,
\bb P_{\sigma} [ H_\zero  < H_{\{\moins, \plus\}} ] 
\;+\; R^{(1)}_\beta 
\end{equation*}
for a new remainder $R^{(1)}_\beta$ whose absolute value is bounded by
$C_0 [\epsilon(\beta) + \delta(\beta)]$. 

By Theorem \ref{mainprop}, $\bb P_{\zero} [ H_{ {\bf +1}} < H_{\bf -1}
]$ converges to $1$. On the other hand, the sum can be handled as in
the proof of Proposition \ref{mt3} to yield that
\begin{equation*}
\lim_{\beta\to\infty}
\frac{\capa ({\bf -1}, {\bf +1})}{\mu_\beta(\sigma^*)\, |\Lambda_L|}\;=\; 
\frac{4(2n_0+1)}3 \;\cdot
\end{equation*}
This completes the proof of the lemma in view of Proposition \ref{mt3}.
\end{proof}

\begin{lemma}
\label{500}
We have that
\begin{equation*}
\lim_{\beta\to\infty} \frac{\capa({\bf +1}, \{{\bf -1}, {\bf 0}\})} 
{\capa({\bf 0}, \{{\bf -1}, {\bf +1}\})} \;=\; 1\;.
\end{equation*}
\end{lemma}

\begin{proof}
By monotonicity of the capacity and by equation (3.5) in \cite{ll},
\begin{equation*}
\capa({\bf +1}, {\bf 0}) \;\le\;
\capa({\bf +1}, \{{\bf -1}, {\bf 0}\}) \;\le\;
\capa({\bf +1}, {\bf 0}) \;+\; \capa({\bf +1}, {\bf -1})\;.
\end{equation*}
We claim that $\capa({\bf -1}, {\bf +1})/\capa({\bf 0} ,{\bf +1}) \to
0$. By Lemma \ref{as64}, we may replace the numerator by $\capa({\bf
  -1}, \{\zero, \plus\})$, and by the second identity of \eqref{m9}, we
may replace the denominator by $\capa(\zero, \{\moins, \plus\})$. At
this point, the claim follows from Proposition \ref{mt3}.

Therefore,
\begin{equation*}
\lim_{\beta\to\infty} \frac{\capa({\bf +1}, \{{\bf -1}, {\bf 0}\})} 
{\capa({\bf 0}, {\bf +1})} \;=\; 1\;.
\end{equation*}
To complete the proof, it remains to recall again the second identity
in \eqref{m9}.
\end{proof}

It follows from the previous result that
\begin{equation}
\label{m10}
\lim_{\beta\to\infty} 
\, \frac{\capa({\bf +1}, \ms M\setminus\{{\bf +1}\})}{\mu_\beta({\bf +1})} 
\, \theta_\beta\;=\; 0\;.    
\end{equation}
Indeed, by the previous lemma, this limit is equal to
\begin{equation*}
\lim_{\beta\to\infty} 
\theta_\beta \, \frac{\capa({\bf 0}, \{{\bf -1}, {\bf +1}\})}{\mu_\beta({\bf 0})}
\, \frac {\mu_\beta({\bf 0})}{\mu_\beta({\bf +1})} \;\cdot
\end{equation*}
This expression vanishes in view of Corollary \ref{as71b} and because
$\mu_\beta({\bf 0})/\mu_\beta({\bf +1}) \to 0$.

\begin{proof}[Proof of Proposition \ref{p71}]
Denote by $r_\beta(\eta,\xi)$ the jump rates of the chain
$\eta_{\theta_\beta t}$. It is enough to prove that
\begin{equation}
\label{lim}
\lim_{\beta\to\infty} r_\beta(\eta,\xi) \;=\; \mb r(\eta,\xi)
\end{equation}
for all $\eta\not =\xi\in\ms M$.

By \cite[Proposition 6.1]{bl2}, the jump rates $r_\beta (\eta,\xi)$,
$\eta\not =\xi\in \ms M$, of the Markov chain $\eta_{\theta_\beta t}$
  are given by
\begin{equation*}
r_\beta(\eta,\xi) \;=\; \theta_\beta \, \lambda(\eta) \, 
\bb P_{\eta} [H_\xi = H^+_{\ms M}]\;. 
\end{equation*}
Dividing and multiplying the previous expression by $\bb P_{\eta}
[H_{\ms M \setminus\{\eta\}} < H^+_{\eta}]$, by definition of the
capacity and by \eqref{probastar}, 
\begin{equation*}
r_\beta(\eta,\xi) \;=\; \frac{\theta_\beta}{\mu_\beta(\eta)} 
\, \capa(\eta, \ms M\setminus\{\eta\})\; 
\bb P_{\eta} [H_\xi < H_{\ms M \setminus \{\eta,\xi\}}]\;. 
\end{equation*}

It follows from this identity and from \eqref{m10} that for $\xi =
{\bf -1}$, ${\bf 0}$,
\begin{equation*}
\lim_{\beta\to\infty} r_\beta({\bf +1},\xi) \;\le\;
\lim_{\beta\to\infty} \frac{\theta_\beta}{\mu_\beta({\bf +1})} 
\, \capa({\bf +1}, \ms M\setminus\{{\bf +1}\}) \;=\; 0\;.
\end{equation*}
On the other hand, by Corollary \ref{as71b}, 
\begin{equation*}
\lim_{\beta\to\infty} \frac{\theta_\beta}{\mu_\beta({\bf 0})} 
\, \capa({\bf 0}, \ms M\setminus\{{\bf 0}\}) \;=\; 1\;,
\end{equation*}
while, by definition, $\theta_\beta \, \capa({\bf -1}, \ms
M\setminus\{{\bf -1}\})/ \mu_\beta({\bf -1}) =1$.  Furthermore, by
Theorem \ref{mainprop},
\begin{equation*}
\lim_{\beta\to\infty} \bb P_{\bf -1} [H_{\bf +1} < H_{\bf 0}]\;=\;
\lim_{\beta\to\infty} \bb P_{\bf 0} [H_{\bf -1} < H_{\bf +1}]\;=\;
0\;. 
\end{equation*}
This yields \eqref{lim} and completes the proof of the lemma.
\end{proof}

\section{The time spent out of $\ms M$}
\label{sec4}

We prove in this section that the total time spent out of $\ms M$ by
the process $\sigma(t\theta_\beta)$ is negligible.  Unless otherwise
stated, we assume that the hypotheses of Theorem \ref{mainprop} are in
force.

\begin{proposition}
\label{ml6}
Let $\Delta = \Omega_L \setminus \ms M$. For every $\xi \in \ms M$,
$t>0$, 
\begin{equation*}
\lim_{\beta\to\infty} \bb E_\xi \Big[ \int_0^{t} \mb
1\{\sigma(s\theta_\beta) \in \Delta\}\, ds \Big] \;=\; 0\;. 
\end{equation*}
\end{proposition}

The proof of this proposition is divided in several steps. Suppose
that the process starts from $\moins$.  In this case, we divide the
time interval $[0,t]$ in 5 pieces, and we prove that the time spent in
$\Delta$ in each time-interval $[0, H_{\mf B^+}]$, $[H_{\mf B^+},
H_\zero]$, $[H_\zero, H_{\mf B^+_0}]$, $[H_{\mf B^+_0}, H_\plus]$ and
$[H_\plus, \infty)$ is negligible.

The proof of the last step requires the introduction of the valley of
$\plus$ which is slightly different from $\ms V_\moins$ and $\ms
V_\zero$. Denote by $\ms V_\plus$ the valley of $\plus$. This is the
set constituted of all configurations which can be attained from
$\plus$ by flipping $n_0(n_0+1)$ or less spins of $\plus$.  The
boundary of this set, denoted by $\mf B^+_\plus$, is formed by all
configurations which differ from $\plus$ at exactly $n_0(n_0+1)$
sites.  The configuration with minimal energy in $\mf B^+_\plus$ is
the one where $n_0(n_0+1)$ $0$-spins form a $n_0\times
(n_0+1)$-rectangle. Denote the set of these configurations by $\mf
R_\plus$.  Fix $\eta\in \mf R_\plus$ and $\xi\in \mf R^a$ and note
that
\begin{equation}
\label{8-5}
\bb H(\eta) \;-\; \bb H(\plus) \;>\; \bb H(\xi) \;-\; \bb
H(\moins)\;. 
\end{equation}
Thus $\ms V_\plus$ is a deeper valley than $\ms V_\moins$ or $\ms
V_\zero$. 

Indeed, according to \cite[Theorem 2.6]{bl2}, the depth of the valley
$\ms V_\plus$ is given by
$\mu_\beta(\plus)/\capa (\plus, \mf B^+_\plus)$. As in the proof of
\eqref{410}, or by applying the Dirichlet and the Thomson principles,
we have that $\capa (\plus, \mf B^+_\plus)$ is of the order of
$|\Lambda_L|\, \mu_\beta(\eta)$ for $\eta\in \mf R_\plus$. Hence the
depth of the valley $\ms V_\plus$ is of the order of
$\kappa_\beta := e^{[\bb H(\eta) - \bb H(\plus) ]\beta}/|\Lambda_L|$,
and starting from $\plus$,
\begin{equation}
\label{8-f1}
H_{\mf B^+_\plus}/\kappa_\beta \quad\text{converges to an exponential
  random variable.}
\end{equation}
By \eqref{2-1}, \eqref{8-5}, $\theta_\beta/\kappa_\beta \to 0$. Hence,
by \eqref{8-f1}, for all $t>0$,
\begin{equation}
\label{8-6}
\lim_{\beta\to\infty}
\bb P_\plus \big[ H_{\mf B^+_\plus} < t \,\theta_\beta\big] \;=\; 0 \;. 
\end{equation}

We turn to the proof of Proposition \ref{ml6}.  We first show that
conditioned to the valley $\ms V_\moins$, the measure $\mu_\beta$ is
concentrated on the configuration $\moins$. The same argument yields
that this result holds for the pairs $(\zero , \ms V_\zero)$, $(\plus
, \ms V_\plus)$.

\begin{lemma}
\label{p74}
Suppose that \eqref{8-3} holds. Then, there exists a constant $C_0$
such that 
\begin{equation*}
\frac{\mu_\beta(\ms V_\moins \setminus
\{\moins\})} {\mu_\beta (\moins)} \;\le \; C_0\,
|\Lambda_L|\, e^{-2\beta} 
\end{equation*}
for all $\beta\ge C_0$.
\end{lemma}

\begin{proof}
Fix $1\le k\le N=n_0(n_0+1)+1$, and denote by $\ms A_k$ the
configurations in $\ms V_\moins$ with $k$ spins different from
$-1$. The ratio appearing in the statement of the assertion is equal
to
\begin{equation}
\label{8-2}
\sum_{k=1}^N \sum_{\sigma\in \ms A_k} \frac{\mu_\beta(\sigma)}
{\mu_\beta (\moins)} \;\le\; 2^N\, \sum_{k=1}^N \sum_{\sigma\in \ms A^0_k} 
\frac{\mu_\beta(\sigma)} {\mu_\beta (\moins)}\;\cdot
\end{equation}
In this equation, $\ms A^0_k$ represents the configurations in
$\ms V_\moins$ with $k$ spins equal to $0$, and we applied Assertion
\ref{fasB}.

Fix $k<N$ and consider the set $\ms A^0_{k,j}$ of configurations in
$\ms A^0_{k}$ for which the $0$-spins have $j$ connected
components. There are at most $C_0 |\Lambda_L|^j$ of such
configurations, and the energy of one of them, denoted by $\sigma$, is
equal to $\bb H(\moins) - kh + I_{-1,0}(\sigma)$, where
$I_{-1,0}(\sigma)$ represents the size of the interface. By
\eqref{8-1}, $I_{-1,0}(\sigma) \ge I_{-1,0}(\sigma^*) + 2(j-1)$, where
$\sigma^*$ is configuration obtained from $\sigma$ by gluing the
connected components. By \cite[Assertion 4.A]{bl2},
$I_{-1,0}(\sigma^*) \ge 4\sqrt{k}$. Therefore, the previous sum for
$k<N$ is bounded above by
\begin{equation*}
C_0\, \sum_{k=1}^{N-1} e^{-[4\sqrt{k} -kh]\beta } \sum_{j=1}^k
|\Lambda_L|^j \, e^{-2(j-1)\beta}\; .
\end{equation*}
Since $|\Lambda_L| \, e^{-2\beta} \le 1/2$ for $\beta$ sufficiently
large and since
$4\sqrt{k} -kh \ge \min\{4-h , 4\sqrt{N-1} - (N-1)h\}$, the previous
sum is less than or equal to
\begin{equation*}
C_0\, |\Lambda_L|\, \Big( e^{-[4-h]\beta } + e^{-2(n_0-1)\beta } \Big)
\;\le\; C_0\, |\Lambda_L|\, \Big( e^{-[4-h]\beta } + e^{-2\beta} \Big)
\;\le\; C_0\, |\Lambda_L|\, e^{-2\beta} 
\end{equation*}
because $4\sqrt{N-1} - (N-1)h = 4\sqrt{n_0(n_0+1)} - n_0(n_0+1) h \ge
4n_0 - 2(n_0+1) = 2(n_0-1) \ge 2$.

It remains to consider the contribution of the set $\ms A_N$. There
are at most $C_0 \, |\Lambda_L|$ configurations in this set, and each
configuration has the same energy. The contribution of these terms to
the sum \eqref{8-2} is bounded by
$C_0 \, |\Lambda_L| e^{-(2n_0+1) \beta} \le C_0 \, |\Lambda_L| e^{-2
  \beta}$. This completes the proof of the lemma.
\end{proof}

\begin{asser}
\label{p72}
We have that
\begin{equation*}
\lim_{\beta\to\infty} \frac 1{\theta_\beta}\, 
\bb E_\moins \Big[ \int_0^{H_{\mf B^+}} \mb
1\{\sigma(s) \in \Delta\}\, ds \Big] \;=\; 0\;. 
\end{equation*}
\end{asser}

\begin{proof}
As the process $\sigma_t$ is stopped at time $H_{\mf B^+}$, we may
replace $\Delta$ by $\Delta\cap\ms V_\moins$.  By \cite[Proposition
6.10]{bl2}, and by definition of $\theta_\beta$, the expression
appearing in the statement of the lemma is bounded above by
\begin{equation*}
\frac 1{\theta_\beta}\,  
\frac{\mu_\beta (\Delta\cap \ms V_\moins)}{\capa (\moins, \mf  B^+)} \;=\; 
\frac{\capa (\moins, \ms M \setminus \{\moins\})}{\capa (\moins, \mf  B^+)} 
\,\frac{\mu_\beta(\Delta\cap \ms V_\moins)}{\mu_\beta (\moins)}\;\cdot
\end{equation*}
By Proposition \ref{mt3} and \eqref{410}, the first ratio converges to
$1/3$, while by Lemma \ref{p74} the second one converges to $0$.
\end{proof}

A similar result holds for the pairs $(\zero, \ms V_\zero)$,
$(\plus, \ms V_\plus)$, where the valleys are defined analogously as
$\ms V_{\moins}$.

Denote by $\mf R^l_2$, resp, $\mf R^s_2$, the set of configurations
with $n_0(n_0+1)+2$ $0$-spins in a background of $-1$-spins. The
$0$-spins form a $[n_0\times (n_0+1)]$-rectangle with two extra
contiguous $0$-spins attached to one of the longest, resp. shortest,
sides of the rectangle. The next result is needed in the proof of
Lemma \ref{p73}.

\begin{lemma}
\label{p75}
For every $t>0$, 
\begin{equation*}
\lim_{\beta\to\infty}
\frac 1{\theta_\beta}\, \max_{\xi \in \mf R^s_2 \cup \mf R} 
\, \bb E_\xi \big[ H_\moins \wedge t \theta_\beta \,\big] \;=\; 0\;. 
\end{equation*}
\end{lemma}

\begin{proof}
Consider a configuration $\sigma\in\mf R$. Applying Lemma \ref{ml8}
repeatedly yields that $\bb P_\sigma [ H_{\mf B^+} < H_\moins ]\to 0$. We may
therefore include the indicator of the set $\{H_\moins < H_{\mf
  B^+}\}$ inside the expectation appearing in the statement of the
lemma. After this inclusion, we may replace $H_\moins$ by $H_{\mf B^+
  \cup \{\moins\}}$.  At this point, it remains to estimate
$(1/\theta_\beta)\, \bb E_\sigma [ H_{\mf B^+ \cup \{\moins\}}]$.
Since the process is stopped as it reaches $\mf B^+$, we may replace
$\sigma_t$ by the reflected process at $\ms V_\moins$. After this
replacement, bound $H_{\mf B^+ \cup \{\moins\}}$ by $H_\moins$.

We need therefore to estimate
$(1/\theta_\beta)\, \bb E^{\ms V}_\sigma [ H_\moins]$.
By \cite[Proposition 6.10]{bl2} and by definition of $\theta_\beta$,
this expression is bounded by
\begin{equation*}
\frac{\capa(\sigma^*, \moins)}{\mu_\beta(\moins)}
\, \frac{1}{\capa_{\ms V}(\sigma, \moins)} \;=\;
\frac{\capa_{\ms V}(\sigma^*, \moins)}{\mu_{\ms V} (\moins)}
\, \frac{1}{\capa_{\ms V}(\sigma, \moins)} \;,
\end{equation*}
where $\sigma^*$ is a configuration in $\mf R^a$. By Lemma \ref{p74},
$\mu_{\ms V} (\moins)\to 1$, while by the proofs of Lemma \ref{ml4}
and \ref{fl1}, 
\begin{equation*}
\lim_{\beta\to\infty} \frac{\capa_{\ms V}(\sigma^*,
  \moins)}{\capa_{\ms V}(\sigma, \moins)}
\;\le \; C_0\, \lim_{\beta\to\infty} \frac{\mu_{\ms V}(\sigma^*)}
{\mu_{\ms V}(\sigma)} \;=\; 0\;.
\end{equation*}

Consider now a configuration $\sigma \in \mf R^s_2$. Denote by $A$ the
$[n_0\times (n_0+2)]$-rectangle obtained from the set of $0$-spins of
$\sigma$ by completing the line or the row where the two extra spins
sit. Denote by $\xi$ the configuration where each site in $A$ has a
$0$-spin, all the other ones being $-1$.

As in the proof of Lemma \ref{ml8}, we define the valley $\ms V_\xi$
of $\xi$ in two stages.  Fix $0\le k\le n_0$. We first flip
sequentially $k$ spins of $A$ from $0$ to $-1$. At each step we only
flip a $0$-spin if it is surrounded by two $-1$-spins. The set of all
configurations obtained by such a sequence of $k$ flips is represented
by $\ms G_k$. In particular, since at the beginning we may only flip
the corners of $A$, $\ms G_1$ is composed of the four configurations
obtained by flipping to $-1$ one corner of $A$. Let
$\ms G = \cup_{0\le k<n_0} \ms G_k$, and note that $\ms G_{n_0}$ has
not been included in the union.

The second stage in the construction of the valley $\ms V_\xi$
consists in flipping a spin of a configuration in $\ms G$. More
precisely, denote by $\ms B$ all configurations which are not in
$\ms G$, but which can be obtained from a configuration in $\ms G$ by
flipping one spin. The set $\ms B$ is interpreted as the boundary of
the valley $\ms V_\xi := \ms G \cup \ms B$ and it contains
$\ms G_{n_0}$.

Note that $\sigma$ belongs to $\ms G$ and that starting from $\sigma$
the process hits $\ms B$ before $\moins$, so that
$H_\moins = H_{\ms B} + H_\moins \circ H_{\ms B}$. Since $(a+b) \wedge
t \le a + (b\wedge t)$, $a$, $b>0$, 
\begin{equation*}
\bb E_\sigma \big[ H_\moins \wedge t \theta_\beta \,\big] \;\le\;
\bb E_\sigma \big[ H_{\ms B} \,\big]  \;+\;
\bb E_\sigma \big[ \,(H_\moins \circ H_{\ms B}) \wedge t \theta_\beta
\,\big]\;. 
\end{equation*}
Replacing $\sigma_t$ by the process reflected at $\ms V_\xi$, applying
\cite[Proposition 6.10]{bl2}, and estimating the capacities yield
that the first term divided by $\theta_\beta$ converges to $0$ as
$\beta\to\infty$. 

We turn to the second term. We may insert the indicator function of
the set $\{H_{\ms B} = H_{\{\eta^{(1)}, \eta^{(2)}\}}\}$, where
$\eta^{(1)}$, $\eta^{(2)}$ are the configurations obtained from $\xi$
by flipping to $-1$ a line or a row of length $n_0$ of the rectangle
$A$. After this insertion, the strong Markov property yields that the
second term of the previous displayed equation is bounded by
\begin{equation*}
\bb E_\sigma \Big[  \, \mb 1\{H_{\ms B} = H_{\{\eta^{(1)},
  \eta^{(2)}\}}\}\, \bb E_{\sigma(H_{\ms B})} 
\big[ H_\moins \wedge t \theta_\beta \,\big]\, \Big]\;.
\end{equation*}
Since the configurations $\eta^{(1)}$, $\eta^{(2)}$ belong to $\mf R$,
the result follows from the first part of the proof.
\end{proof}

\begin{lemma}
\label{p73}
For every $t>0$, 
\begin{equation*}
\limsup_{\beta\to\infty}
\frac 1{\theta_\beta}\, \bb E_\moins \Big[ \int_0^{t\theta_\beta} \mb
1\{\sigma(s) \in \Delta\}\, ds \Big] \;\le\; 3\, \limsup_{\beta\to\infty}
\max_{\xi \in \mf R^l_2}\, \frac 1{\theta_\beta}\, 
\bb E_\xi \Big[ \int_0^{t\theta_\beta} \mb 1\{\sigma(s) \in \Delta\}\, ds \,\Big]
\;. 
\end{equation*}
\end{lemma}

\begin{proof}
By Assertion \ref{p72},
\begin{equation*}
\lim_{\beta\to\infty} 
\frac 1{\theta_\beta}\, \bb E_\moins \Big[ \int_0^{t\theta_\beta} \mb
1\{\sigma(s) \in \Delta\}\, ds \, \mb 1 \big\{ H_{\mf B^+} \ge t\theta_\beta
\big\}\, \Big] \;=\; 0\;.
\end{equation*}
We may therefore insert the indicator of the set
$\{ H_{\mf B^+} < t\theta_\beta\}$ inside the expectation appearing in
the statement of the assertion at a negligible cost. By Proposition
\ref{Lemma1}, we may also insert the indicator of the set
$\{ H_{\mf B^+} = H_{\mf R^a} \}$. After inserting these indicator
functions, by the strong Markov property, we get that the expectation
appearing in the statement of the lemma is bounded by
\begin{equation*}
\frac 1{\theta_\beta}\, \bb E_\moins \Big[ 
\mb 1 \big\{H_{\mf R^a} = H_{\mf B^+} \le t\theta_\beta \big\}
\, \bb E_{\sigma(H_{\mf B^+})}  \Big[ \int_0^{t\theta_\beta} \mb
1\{\sigma(s) \in \Delta\}\, ds \, \Big]\, \Big] \;+\; R_\beta\;,
\end{equation*}
where $R_\beta\to 0$.

The previous expectation is bounded by
\begin{equation*}
\frac 1{\theta_\beta}\, \sum_{\sigma\in \mf R^a}
\bb P_\moins \big[ H_{\sigma} = H_{\mf B^+} \big]
\, \bb E_\sigma \Big[ \int_0^{t\theta_\beta} \mb
1\{\sigma(s) \in \Delta\}\, ds \,\Big] \;.
\end{equation*}
Since $\lambda_\beta(\sigma) \ge 1$ for $\sigma\in\mf R^a$ and since
$\mf R^a\subset \Delta$, by removing from the integral the interval
$[0,\tau_1]$, where $\tau_1$ represents the time of the first jump,
the previous expectation is less than or equal to
\begin{equation*}
\frac 1{\theta_\beta}\, \sum_{\sigma\in \mf R^a} \sum_\xi
\bb P_\moins \big[ H_{\sigma} = H_{\mf B^+} \big]\, p_\beta(\sigma, \xi)
\, \bb E_\xi \Big[ \int_0^{t\theta_\beta} \mb
1\{\sigma(s) \in \Delta\}\, ds \,\Big] \;+\; \frac 1{\theta_\beta}\;\cdot
\end{equation*}
By Assertions \ref{mas1} and \ref{mas2}, we may disregard all
configurations $\xi$ which do not have two contiguous attached
$0$-spins and which do not belong to $\mf R$. The previous expression
is thus bounded above by
\begin{equation*}
\max_{\xi \in \mf R^l_2}\, \frac 1{\theta_\beta}\, 
\bb E_\xi \Big[ \int_0^{t\theta_\beta} \mb 1\{\sigma(s) \in \Delta\}\, ds \,\Big]
\;+\; \frac 23\, \frac 1{\theta_\beta}\, \max_{\xi \in \mf R^s_2 \cup \mf R} 
\, \bb E_\xi \Big[ \int_0^{t\theta_\beta} \mb
1\{\sigma(s) \in \Delta\}\, ds \,\Big] \;+\; R_\beta\;,
\end{equation*}
where $\mf R^s_2$ represents the set of configurations with two
contiguous $0$-spins attached to the shortest side of the rectangle,
and $R_\beta \to 0$.

The second expectations is bounded by
\begin{equation*}
\frac 1{\theta_\beta}\, \max_{\xi \in \mf R^s_2 \cup \mf R} 
\, \bb E_\xi \big[ H_\moins \wedge t \theta_\beta \,\big] \;+\; 
\frac 23\, \frac 1{\theta_\beta}\, 
\, \bb E_\moins \Big[ \int_0^{t\theta_\beta} \mb
1\{\sigma(s) \in \Delta\}\, ds \,\Big]\;.
\end{equation*}
By Lemma \ref{p75}, the first term vanishes as $\beta\to\infty$.  The
second one can be absorbed in the left-hand side of the expression
appearing in the statement of the lemma, which completes the proof of
the lemma.
\end{proof}

\begin{lemma}
\label{p76}
For every $t>0$, 
\begin{equation*}
\limsup_{\beta\to\infty}
\max_{\xi \in \mf R^l_2}\, \frac 1{\theta_\beta}\, 
\bb E_\xi \big[ H_\zero \wedge t\theta_\beta \,\big]
\;=\; 0 \;. 
\end{equation*}
\end{lemma}

\begin{proof}
The proof of this lemma follows the steps of Section \ref{proofs1}
where we described the growth of the supercritical droplet. Fix a
configuration $\xi$ in $\mf R^l_2$ and recall Lemma \ref{mas3}. 
By this result,
\begin{equation*}
\frac 1{\theta_\beta}\, \bb E_\xi \big[ \, H_\zero \wedge t\theta_\beta
\,\big] \;\le\; \frac 1{\theta_\beta}\, \bb E_\xi \big[ \,
\mb 1\{E_{c,d} = H_{c,d} \}
\, (H_\zero \wedge t\theta_\beta) \,\big] \;+\; t\, \delta_1 (\beta)\;. 
\end{equation*}
Since $E_{c,d} < H_\zero$, we may write $H_\zero$ as $E_{c,d} +
H_\zero \circ E_{c,d}$ and bound $H_\zero \wedge t\theta_\beta$ by
$E_{c,d} + [(H_\zero \circ E_{c,d}) \wedge t\theta_\beta]$. Hence, by
the strong Markov property, the previous expression is less than or
equal to
\begin{equation*}
t\, \delta_1 (\beta) \;+\; \frac 1{\theta_\beta}\, 
\bb E_\xi \big[  E_{c,d} \,\big] \;+\; 
\frac 1{\theta_\beta}\, \bb E_\xi \Big[ 
\mb 1\{E_{c,d} = H_{c,d} \}\, \bb E_{\sigma(E_{c,d})} \big[ 
\, H_\zero \wedge t\theta_\beta \big] \,\Big] \;.
\end{equation*}
On the set $\{E_{c,d} = H_{c,d} \}$, $\sigma(E_{c,d})$ is a
configuration with $(n_0+1)^2$ $0$-spins forming a square in a sea of
$-1$-spins. The previous expression is thus bounded by
\begin{equation*}
t\, \delta_1 (\beta) \;+\; \frac 1{\theta_\beta}\, 
\bb E_\xi \big[  E_{c,d} \,\big] \;+\; 
\frac 1{\theta_\beta}\, \max_{\eta} \bb E_\eta \big[ 
\, H_\zero \wedge t\theta_\beta \,\big] \;,
\end{equation*}
where the maximum is carried over the configurations just described.

Compare the previous expression with the one on the statement of the
lemma. We replaced the set $\mf R^l_2$ by the set of $0$-squares of
length $(n_0+1)$ at the cost of the error term $t\, \delta_1 (\beta)$,
and the expectation of the hitting time of the boundary of the
neighborhood of $\xi$. Proceeding in this way until hitting $\zero$
will bring the sum of all errors and the sum of the expectations of
all hitting times. The assumptions of Theorem \ref{mainprop} were
inserted to guarantee that the sum of the error terms converge to $0$
as $\beta\to\infty$. The hitting times of the boundary of the
neighborhoods involve either the creation of two contiguous $0$-spin
at the boundary of a rectangle, whose time length is of order
$e^{(2-h)\beta}$, or the filling of a side of a rectangle, which
corresponds to the hitting time of an asymmetric one-dimensional
random walk, whose time-length order is proportional to the length of
the rectangle. Both orders are much smaller than $\theta_\beta$, which
completes the proof of the lemma.
\end{proof}

\begin{lemma}
\label{p77}
For every $t>0$, 
\begin{equation*}
\limsup_{\beta\to\infty}
\frac 1{\theta_\beta}\, \bb E_\plus \Big[ \int_0^{t\theta_\beta} \mb
1\{\sigma(s) \in \Delta\}\, ds \Big] \;=\; 0 \;. 
\end{equation*}
\end{lemma}

\begin{proof}
By \eqref{8-6}, we may insert the indicator of the set $\{H_{\mf
  B^+_\plus} \ge t \,\theta_\beta\}$ inside the expectation. After
this insertion, we may replace the process $\sigma_t$ by the reflected
process at $\ms V_\plus$, denoted by $\eta_t$, and then bounded the
expression by
\begin{equation*}
\frac 1{\theta_\beta}\, \bb E^{\ms V}_\plus \Big[ \int_0^{t\theta_\beta} \mb
1\{\eta(s) \in \Delta\}\, ds \Big] \;. 
\end{equation*}
This term is equal to
\begin{equation*}
\frac 1{\theta_\beta}\, \int_0^{t\theta_\beta} \bb P^{\ms V}_\plus \big[ \mb
\eta (s) \in \Delta \big] \, ds \;\le\;
\frac 1{\theta_\beta}\, \int_0^{t\theta_\beta} \frac 1 {\mu_{\ms
    V}(\plus)} \sum_{\xi\in \ms V_\plus} \mu_{\ms V} (\xi) \, \bb P^{\ms V}_\xi \big[ \mb
\eta (s) \in \Delta \big] \, ds \;. 
\end{equation*}
As $\mu_{\ms V}$ is the stationary state for the reflected process,
this expression is equal to
\begin{equation*}
t\, \frac {\mu_{\ms V} (\Delta)}  {\mu_{\ms V}(\plus)}  \;=\;
t\, \frac {\mu_{\beta} (\ms V_\plus \setminus \{\plus\})}
{\mu_{\beta}(\plus)} \;\cdot 
\end{equation*}
By Lemma \ref{p74}, for $\ms V_\plus$ instead of $\ms V_\moins$, this
expression vanishes as $\beta\to\infty$. 
\end{proof}

\begin{proof}[Proof of Proposition \ref{ml6}]
Lemmata \ref{p73} and \ref{p76} show that the time spent on $\Delta$
until the process reaches $\zero$ is negligible. We may repeat the
same argument to extend the result up to the time where the process
reaches $\plus$.  Once the process reaches $\plus$, we apply Lemma
\ref{p77}. 
\end{proof}

\section{Proof of Theorem \ref{mt1b}}
\label{sec6}

Unless otherwise stated, we assume throughout this section that the
hypotheses of Theorem \ref{mt1b} are in force.  According to
\cite[Proposition 2.1]{llm}, we have to show that for all $\eta\in\ms
M$,
\begin{equation}
\label{9-2}
\lim_{\delta\to 0} \limsup_{\beta\to\infty} 
\sup_{2\delta\le s\le 3\delta} 
\bb P_{\eta} [ \sigma(s\theta_\beta) \in \Delta ]\;=\;0\;. 
\end{equation}
We present the proof for $\eta = \moins$. The proof for $\eta=\zero$
is identical. The one for $\eta=\plus$ is even simpler because the
valley is deeper.

\begin{lemma}
\label{l9-1}
Under $\bb P_\moins$, the random variable $3\, H_{\mf
  B^+}/\theta_\beta$ converges to a mean one exponential random
variable.
\end{lemma}

\begin{proof}
As we are interested in $H_{\mf B^+}$, we may replace the process
$\sigma_t$ by the reflected process at $\ms V_\moins$, denoted by
$\eta_t$. 

The proof is based on \cite[Theorem 2.6]{bl2} applied to the triple
$(\{\moins\},\ms V_\moins \setminus \mf B^+, \moins)$. We claim that
the process $\eta_t$ fulfills all the hypotheses of this
theorem. Condition (2.14) is satisfied because the set $\{\moins\}$ is
a singleton, and condition (2.15) is in force in view of Lemma
\ref{p74}. Therefore, by this theorem, the triple $(\{\moins\},\ms
V_\moins \setminus \mf B^+, \moins)$ is a valley of depth
\begin{equation*}
\frac{\mu_{\ms V} (\moins)}{\capa_{\ms V} (\moins, \mf B^+)} \;=\; 
\frac{\mu_{\beta} (\moins)}{\capa (\moins, \mf B^+)} \;=\; 
\frac{\capa (\moins, \{\zero, \plus\})}{\capa (\moins, \mf B^+)}\, \theta_\beta \;,
\end{equation*}
where the last identity follows from the definition of $\theta_\beta$
given in \eqref{71}. By \eqref{410} and Proposition \ref{mt3}, the
last ratio converges to $1/3$. Hence, by property (V2) of
\cite[Definition 2.1]{bl2}, $3\, H_{\mf B^+}/ \theta_\beta$ converges
to a mean-one exponential random variable, as claimed.
\end{proof}

By the previous lemma, for $2\delta \le s\le 3\delta$,
\begin{equation*}
\bb P_{\moins} \big[\, \sigma(s\theta_\beta) \in \Delta \big] \;=\;
\bb P_{\moins} \big[\, \sigma(s\theta_\beta) \in \Delta \,,\, H_{\mf
  B^+} \ge 4\delta \theta_\beta \big] \;+\; R_{\beta, \delta}\;,
\end{equation*}
for some remainder $R_{\beta, \delta}$ which vanishes as
$\beta\to\infty$ and then $\delta\to 0$. On the set $\{H_{\mf B^+} \ge
4\delta\}$, we may replace the process $\sigma_t$ by the reflected
process $\eta_t$ and bound the first term by
\begin{align*}
\bb P^{\ms V}_{\moins} \big[\, \eta (s\theta_\beta) \in \Delta \big] 
\; \le \; \frac 1{\mu_{\ms V}(\moins)}\, \sum_{\sigma\in \ms V_\moins}
\mu_{\ms V}(\sigma)  \, \bb P^{\ms V}_{\sigma} \big[\,
\eta (s\theta_\beta) \in \Delta \big] \;
= \; \frac {\mu_{\ms V}(\Delta)}{\mu_{\ms V}(\moins)}
\end{align*}
because $\mu_{\ms V}$ is the stationary state. By Lemma \ref{p74} this
expression vanishes as $\beta\to\infty$, which proves \eqref{9-2} for
$\eta=\moins$. 

\smallskip\noindent{\bf Acknowledgements.} 
C. Landim has been partially supported by FAPERJ CNE E-26/201.207/2014
and by CNPq Bolsa de Produtividade em Pesquisa PQ
303538/2014-. C. Landim and M. Mourragui have been partially supported
by ANR-15-CE40-0020-01 LSD of the French National Research Agency.

\end{document}